\documentclass[a4paper,11pt]{amsart}
\usepackage[colorlinks, linkcolor=blue,anchorcolor=Periwinkle,
    citecolor=blue,urlcolor=Emerald]{hyperref}
\usepackage[all]{xy}
\SelectTips{cm}{}

\usepackage{graphicx}
\usepackage{psfrag}

\usepackage{mathtools}
\usepackage{tikz}
\usepackage{tikz-cd}
\usepackage{adjustbox}
\usepackage{extarrows}

\usepackage{amssymb}
\usetikzlibrary{decorations.pathreplacing}
\usetikzlibrary{matrix,arrows}

\usepackage{enumitem}

\usepackage{stmaryrd}

\usepackage{quiver}

\newcommand{\ie}{{i.e.}\ }
\newcommand{\cf}{{cf.}\ }

\newcommand{\ko}{\: , \;}

\newcommand{\ol}[1]{\overline{#1}}

\numberwithin{equation}{subsection}

\newtheorem{classification-theorem}[subsection]{Classification Theorem}
\newtheorem{decomposition-theorem}[subsection]{Decomposition Theorem}
\newtheorem{proposition-definition}[subsection]{Proposition-Definition}
\newtheorem{periodicity-conjecture}[subsection]{Periodicity Conjecture}

\newtheorem{theorem}{Theorem}
\numberwithin{theorem}{subsection}
\newtheorem{thmx}{Theorem}

\newtheorem{lemma}[theorem]{Lemma}
\newtheorem{proposition}[theorem]{Proposition}
\newtheorem{corollary}[theorem]{Corollary}

\newtheorem{example}[theorem]{Example}
\newtheorem{remark}[theorem]{Remark}
\newtheorem{assumptions}[theorem]{Assumptions}

\newcommand{\reminder}[1]{}

\newcommand{\Mod}{\mathrm{Mod}\,}

\newcommand{\per}{\mathrm{per}\,}
\newcommand{\pvd}{\mathrm{pvd}\,}
\newcommand{\add}{\mathrm{add}\,}
\newcommand{\sub}{\mathrm{sub}\,}
\newcommand{\op}{^{op}}

\newcommand{\Tr}{\mathrm{Tr}}
\newcommand{\tr}{\mathrm{tr}}
\newcommand{\PC}{\mathrm{PC}\,}

\newcommand{\scr}{\mathcal}

\newcommand{\cone}{\mathrm{cone}\,}
\newcommand{\cocone}{\mathrm{cocone}\,}

\newcommand{\Z}{\mathbb{Z}}

\newcommand{\C}{\mathbb{C}}
\newcommand{\R}{\mathbb{R}}
\newcommand{\D}{\mathbb{D}}

\newcommand{\iso}{\xrightarrow{_\sim}}
\newcommand{\liso}{\xleftarrow{_\sim}}
\newcommand{\id}{\mathbf{1}}

\newcommand{\Hom}{\mathrm{Hom}}

\newcommand{\RHom}{\mathrm{RHom}}

\newcommand{\Ext}{\mathrm{Ext}}

\newcommand{\End}{\mathrm{End}}
\newcommand{\rad}{\mathrm{rad}\,}

\newcommand{\ten}{\otimes}
\newcommand{\lten}{\overset{\boldmath{L}}{\ten}}

\newcommand{\ca}{{\mathcal A}}
\newcommand{\cb}{{\mathcal B}}
\newcommand{\cc}{{\mathcal C}}
\newcommand{\cd}{{\mathcal D}}
\newcommand{\ce}{{\mathcal E}}

\newcommand{\ch}{{\mathcal H}}

\newcommand{\cm}{{\mathcal M}}

\newcommand{\cp}{{\mathcal P}}

\newcommand{\ct}{{\mathcal T}}

\newcommand{\bp}{\mathbf{p}}

\newcommand{\ba}{\mathbf{a}}

\newcommand{\Ga}{\Gamma}
\newcommand{\La}{\Lambda}
\newcommand{\Si}{\Sigma}
\newcommand{\si}{\sigma}

\newcommand{\wg}{\wedge}

\newcommand{\eps}{\varepsilon}
\renewcommand{\phi}{\varphi}

\renewcommand{\hat}[1]{\widehat{#1}}

\renewcommand{\tilde}[1]{\widetilde{#1}}

\begin{document}

\date{September 2, 2024}

\title{On Amiot's conjecture}

\author{Bernhard Keller}
\address{Universit\'e Paris Cit\'e and Sorbonne Université, CNRS, IMJ-PRG, F-75013 Paris, France}
\email{bernhard.keller@imj-prg.fr}
\urladdr{https://webusers.imj-prg.fr/~bernhard.keller/}

\author{Junyang Liu}
\address{Universit\'e Paris Cit\'e and Sorbonne Université, CNRS, IMJ-PRG, F-75013 Paris, France 
and Yau Mathematical Sciences Center, Tsinghua University, Beijing 100084, China}
\email{liuj@imj-prg.fr}
\email{liujunya18@mails.tsinghua.edu.cn}
\urladdr{https://webusers.imj-prg.fr/~junyang.liu}

\begin{abstract}
In a survey paper in 2011, Amiot proposed a conjectural
characterisation of the cluster categories which were conceived in the mid 2000s to
lift the combinatorics of Fomin--Zelevinsky's cluster algebras to the categorical
level. This paper is devoted to a proof of (a variant of) her conjecture. More generally,
cluster categories admit higher-dimensional and relative variants, the so-called Higgs categories
recently introduced by Wu. We also prove higher-dimensional and relative variants
of the conjecture.
\end{abstract}




\maketitle

\vspace*{-1cm}
\tableofcontents

\section{Introduction}

\subsection{Overview}

We prove (a variant of) a conjecture formulated by Amiot \cite{Amiot11}
in 2011 in the form of a question. The conjecture characterises the {\em cluster categories} (and their higher-dimensional analogues), which have appeared \cite{MarshReinekeZelevinsky03, CalderoChapotonSchiffler06, BuanMarshReinekeReitenTodorov06, Keller05} in the mid 2000s in the theory of (additive) categorification of cluster algebras invented by Fomin--Zelevinsky \cite{FominZelevinsky02} around the year 2000. One of the main points in the proof is a lift \cite{KellerLiu23a} to the differential graded level of Amiot's construction \cite{Amiot09} of Calabi--Yau structures on Verdier quotients. Cluster categories also admit `relative' variants recently introduced by Wu \cite{Wu23} and known as {\em Higgs categories}. We also prove a relative variant of Amiot's conjecture.

\subsection{Cluster algebras and additive categorification}

Let us recall that a cluster algebra is a commutative algebra endowed with a family
of distinguished generators, the {\em cluster variables}, which are assembled into
finite sets of fixed cardinality called the {\em clusters}. The clusters are constructed
recursively starting from the datum of a {\em quiver} (=oriented graph). We refer
to \cite{FominWilliamsZelevinsky16, FominWilliamsZelevinsky17} and the references
given there for more information on cluster algebras. 
In additive categorification, we aim at constructing, from the quiver $Q$ of
a cluster algebra $\ca$, a triangulated category $\cc_Q$ and a decategorification map $\chi\colon \cc_Q \to \ca$ which associates elements
of $\ca$ with the objects in $\cc_Q$ so as to establish as close a correspondence as possible
between the combinatorics of the indecomposable rigid objects in
$\cc_Q$ and those of the cluster variables in $\ca$. For an acyclic quiver
$Q$, the construction of the corresponding {\em cluster category} $\cc_Q$
is due to Buan--Marsh--Reineke--Reiten--Todorov
\cite{BuanMarshReinekeReitenTodorov06}. As shown by work of 
Derksen--Weyman--Zelevinsky \cite{DerksenWeymanZelevinsky08},
in the general case of a quiver admitting oriented cycles, in addition
to the quiver $Q$, we need to consider a (non-degenerate) {\em potential},
\ie a formal linear combination  $W$ of cycles of $Q$. In this more general
case, the construction of the corresponding cluster category  $\cc_{Q,W}$ is due
to Amiot \cite{Amiot09} and that of the decategorification map
$\chi\colon \cc_{Q,W} \to \ca$ to Caldero--Chapoton \cite{CalderoChapoton06},
Palu \cite{Palu08a}, Plamondon \cite{Plamondon11} and others, \cf also \cite{Plamondon18}
and \cite{Keller10b, Keller12a}.
These cluster categories also admit natural higher-dimensional analogues
first studied by Guo \cite{Guo11} and later by several other authors, \cf for example
\cite{ThanhofferVandenBergh16, KalckYang16, IyamaYang18, Pressland17, IkedaQiu23, Jin23, Hanihara22a, Hanihara22}.

More generally, a cluster algebra is associated with an {\em ice quiver}, \ie a
quiver endowed with a (not necessarily full) subquiver whose vertices and
arrows are called {\em frozen}. The cluster variables associated with the
frozen vertices then occur in all the clusters. They are also called {\em coefficients}.
This more general class of cluster algebras is important for at least two reasons:
\begin{itemize}
\item[a)] cluster algebras appearing as coordinate algebras of algebraic
varieties (like Grassmannians, unipotent cells, Schubert cells, \ldots)
always have coefficients,
\item[b)] frozen vertices and arrows are crucial when we want to
{\em glue} ice quivers (respectively cluster algebras), an operation which
is of great importance for example in Fock--Goncharov's amalgamation
construction in their approach to higher Teichm\"uller theory
\cite{FockGoncharov06}.
\end{itemize}
The most successful approach to the categorification of cluster algebras
with coefficients is due to Geiss--Leclerc--Schr\"oer, \cf for example 
\cite{GeissLeclercSchroeer05, GeissLeclercSchroeer06,GeissLeclercSchroeer12,
GeissLeclercSchroeer13}. Instead of triangulated
categories, they use Frobenius exact categories whose associated stable
categories are ($\Hom$-finite and) $2$-Calabi--Yau as triangulated categories.
They construct them as extension-closed subcategories of categories of
finite-dimensional modules over preprojective algebras. As they show,
large classes of examples of cluster algebras from Lie theory admit
such {\em Frobenius exact categorifications}. On the other hand, 
it is easy to see that many simple ice quivers cannot admit
a Frobenius exact categorification (for example any non-empty
acyclic quiver with principal coefficients). In \cite{Wu23, Wu21, KellerWu23},
Wu and Keller--Wu have extended Geiss--Leclerc--Schr\"oer's
approach to (almost) all ice quivers with potential $(Q,F,W)$ by
introducing the {\em Higgs category} $\ch_{Q,F,W}$, a certain extriangulated
category in the sense of Nakaoka--Palu \cite{NakaokaPalu19} which generalises
the Frobenius exact categories used by Geiss--Leclerc--Schr\"oer. 
We refer to \cite{Christ22, Christ22a} for a slightly different categorical approach.
As shown by Wu \cite{Wu23}, Higgs categories also admit natural
higher-dimensional analogues.

\subsection{The characterising properties of cluster categories}

The cluster category $\cc_Q$ associated with a finite acyclic quiver $Q$ is a Karoubian, ($\Hom$-finite,) $2$-Calabi--Yau algebraic triangulated category with a canonical cluster-tilting object. In 2008, Keller--Reiten \cite{KellerReiten08} proved that a $\Hom$-finite triangulated category is triangle equivalent to the cluster category associated with a finite acyclic quiver $Q$ if and only if it is $2$-Calabi--Yau, algebraic and contains a cluster-tilting object with endomorphism algebra isomorphic to $kQ$. It is natural to try and extend this theorem to the case of non-acyclic quivers. In 2009, Amiot \cite{Amiot09} constructed the {\em generalised cluster category} $\cc_{Q,W}$ associated with a quiver with potential $(Q,W)$. She showed that for an acyclic quiver with the zero potential, the generalised cluster category $\cc_{Q,0}$ is equivalent to the classical cluster category $\cc_Q$. She proved that if $(Q,W)$ is Jacobi-finite, then $\cc_{Q,W}$ is Karoubian, ($\Hom$-finite,) $2$-Calabi--Yau and contains a canonical cluster-tilting object with endomorphism algebra isomorphic to the Jacobian algebra. In Question~2.20 of \cite{Amiot11}, she asked whether, if the ground field is of characteristic~$0$ and algebraically closed, all triangulated
categories with these properties are equivalent to cluster categories $\cc_{Q,W}$. This is a generalisation of the 
theorem proved by Keller--Reiten in the case where the endomorphism algebra of the cluster-tilting object is hereditary. The question is discussed for example in \cite{KalckYang16}. The fact that the answer should be positive is now known as {\em Amiot's conjecture}, \cf for example Conjecture~1.1 in \cite{KalckYang20}.

There are several known classes where it holds: by a theorem of Crawley--Boevey \cite{CrawleyBoevey00}, the stable category of the preprojective algebra associated with a Dynkin quiver is \linebreak $2$-Calabi--Yau. Moreover, Geiss--Leclerc--Schr\"oer \cite{GeissLeclercSchroeer08b} showed that it contains a canonical cluster-tilting object, \cf also \cite{BuanIyamaReitenScott09}. In 2009, Amiot \cite{Amiot09} proved that the conjecture holds for these examples. In \cite{GeissLeclercSchroeer06, BuanIyamaReitenScott09}, \cf also Theorem~3.23 of \cite{AmiotIyamaReitenTodorov12}, 
they have been generalised as follows:
let $Q$ be an acyclic quiver and $w$ an element of the associated Coxeter group. For each vertex $i$ of the quiver $Q$, let $I_i$ be the two-sided ideal of the completed preprojective algebra $\La_Q$ generated by $1_{\La_Q}-e_i$. Let $w=s_{i_1}\ldots s_{i_l}$ be a reduced expression for $w$ and let $I_w=I_{i_1}\ldots I_{i_l}$. Let $\cc_w=\sub(\La_Q/I_w)$ be the full subcategory of the category of finite-dimensional modules over $\La_Q$ formed by the submodules of finite direct sums of copies of $\La_Q/I_w$. Then $\cc_w$ is a Frobenius exact category whose stable category is $2$-Calabi--Yau and contains a canonical cluster-tilting object. Here as well, the conjecture 
holds as shown in \cite{Amiot09, AmiotReitenTodorov11}. 

There are examples from algebraic geometry as well. Let $R$ be a $3$-dimensional Gorenstein local commutative ring with isolated singularity. Then the stable category of finitely generated maximal Cohen--Macaulay modules over $R$ is ($\Hom$-finite and) $2$-Calabi--Yau by an old theorem of Auslander's \cite{Auslander76}. In certain cases \cite{BurbanIyamaKellerReiten08}, it contains a cluster-tilting object. This holds for example when $R$ is the $G$-invariant ring $k[x_1,x_2,x_3]^G$ for a finite cyclic subgroup $G\subseteq SL_3(k)$. Then the conjecture holds, \cf \cite{AmiotIyamaReiten15, ThanhofferVandenBergh18}. Finally, let us point out that in the case where
the endomorphism algebra of the given cluster-tilting object is monomial, the conjecture was
proved by Garcia Elsener \cite{Garcia21}.

In \cite{Guo11}, Guo generalised Amiot's construction from dimension $2$ to arbitrary dimensions $d\geq 2$, \ie she constructed $d$-Calabi--Yau triangulated categories containing $d$-cluster-tilting objects. She associated these with suitable $(d+1)$-Calabi--Yau dg(=differential graded) algebras. This leads to a natural generalisation of Amiot's conjecture to arbitrary dimensions. In this article, we prove a variant of this generalised conjecture. The main difference between the original statement of the conjecture and the variant we prove concerns the given $d$-Calabi--Yau structure: whereas Amiot supposes that it is a structure on the algebraic triangulated category $\cc$, we assume that it is a $d$-Calabi--Yau structure (in the sense of Kontsevich) on a given dg enhancement $\cc_{dg}$ of $\cc$. Another technical but important difference is that we suppose our categories to be enriched over the category of pseudo-compact vector spaces. Indeed, Davison's examples of Calabi--Yau algebras not coming from potentials \cite{Davison12} suggest that without this assumption, the conjecture might not be true. On the other hand, the variant we prove is more general than the original conjecture insofar as we do not assume the ground field to be algebraically closed.

\subsection{Organisation and outline of this paper}

In section~\ref{section:preliminaries}, we recall basic facts on pseudo-compact vector spaces, augmented algebras and extriangulated categories. In section~\ref{section:pseudo-compact dg algebras}, we prove favourable properties of pseudo-compact dg algebras which will be very useful later on: in Proposition~\ref{prop:complete augmented}, we show that up to quasi-isomorphism, a connective pseudo-compact dg algebra $A$ is augmented and complete if this holds for $H^0(A)$. In Proposition~\ref{prop:minimal cofibrant replacement}, we show that pseudo-compact dg modules over complete $l$-augmented pseudo-compact dg algebras admit minimal cofibrant models. This allows one to detect the perfection of a pseudo-compact dg module by looking at $\RHom_A(M,l)$, \cf Corollary~\ref{cor:perfect}. We deduce that a complete augmented pseudo-compact dg algebra is smooth if and only if its Koszul dual is proper if and only if its perfectly valued derived category is contained in its perfect derived category, \cf Corollary~\ref{cor:smooth}. In Lemma~\ref{lem:forgetful functor}, we show that the category of perfectly valued pseudo-compact dg modules over a smooth pseudo-compact dg algebra embeds fully faithfully into the category of perfectly valued dg modules over the underlying non-topological dg algebra. In Lemma~\ref{lem:pseudo-compact localisation}, we give a sufficient condition for a continuous morphism between pseudo-compact dg algebras to be a pseudo-compact localisation. In section \ref{ss:reminders on Calabi--Yau structures}, we recall the construction of cyclic homology and its variants for (pseudo-compact) dg algebras and give the definitions of left and right $d$-Calabi--Yau structures in the absolute and the relative case. It is well-known that a left Calabi--Yau structure on the perfect dg derived category yields a right Calabi--Yau structure on the perfectly valued dg derived category. In Proposition~\ref{prop:left and right}, we show that for complete augmented pseudo-compact dg algebras, this assignment is bijective. In section~\ref{ss:deformed dg preprojective algebras}, we recall the construction of deformed dg preprojective algebras from \cite{VandenBergh15} and the construction of Ginzburg--Lazaroiu morphisms from \cite{KellerLiu23}. In section~\ref{ss:cluster categories and Higgs categories}, we recall the construction of the cluster category associated with a deformed dg preprojective algebra following \cite{Guo11} and the construction of the relative cluster category and the Higgs category associated with a Ginzburg--Lazaroiu morphism following \cite{Wu23, KellerWu23}.

In section~\ref{ss:htpy short exact sequences}, we recall the definition of homotopy short exact sequences of complexes. Such a sequence yields a long exact sequence in the category of abelian groups, where the connecting morphism can be computed using the homotopy snake lemma, \cf Lemma~\ref{lemma:homotopy snake lemma}. For a full dg subcategory $\cb$ of a dg category $\ca$, the Drinfeld quotient \cite{Drinfeld04} is an explicit description of the dg quotient \cite{Keller99}
of $\ca$ by $\cb$. Using this description we construct a homotopy short exact sequence of dg $\ca$-bimodules, 
\cf Proposition~\ref{prop:htpy short exact sequence of resolutions} and a homotopy short exact sequence of Hochschild complexes, \cf Proposition~\ref{prop:htpy short exact sequence of Hochschild complexes}. In Lemma~\ref{lem:contracting homotopy}, we recall Kontsevich's characterisation of homotopy isomorphisms using contracting homotopies for their mapping cones in a pretriangulated dg category. Then we use these tools to give a (partial but sufficiently precise) description of the connecting morphism in Hochschild homology induced by an exact sequence of dg categories.

\begin{thmx}[=Theorem~\ref{thm:connecting morphism}] \label{thm:A}
Let $\cc=\ca/\cb$ be the Drinfeld quotient of a pretriangulated dg category $\ca$ by a pretriangulated full dg subcategory $\cb$. Let $M$ be the dg $\ca$-bimodule determined by $M(X,Y)=\ca(\nu X,Y)$ associated with a dg endofunctor $\nu$ of $\ca$ taking $\cb$ to itself. Then the connecting morphism
\[
\delta\colon HH_1(\cc,M^\cc)\longrightarrow HH_0(\cb,M_\cb)\: ,
\]
where $M_\cb$ is the restriction of $M$ to $\cb^e$ and $M^\cc$ the induction of $M$ to $\cc^e$, admits a (partial) explicit description.
\end{thmx}

In particular, we lift Amiot's construction of Calabi--Yau structures on certain Verdier quotients to the level of dg enhancements, \cf Remark~\ref{rk:compatibility}. This result was previously announced in \cite{KellerLiu23a}.

For an exact pseudo-compact category with enough projective objects, we show that certain dg subquotients of its dg category of bounded complexes can be endowed with natural structures of pseudo-compact dg categories, \cf Proposition~\ref{prop:pseudo-compact dg categories}. In Theorem~\ref{thm:main}, we state the structure theorem in the absolute case. This result was previously announced in \cite{KellerLiu23a}.

\begin{thmx}[=Theorem~\ref{thm:main}] \label{thm:B}
Under suitable assumptions, for a triangulated category $\cc$ and a positive integer $d$, the following are equivalent.
\begin{itemize}
\item[i)] $\cc$ is triangle equivalent to the cluster category $\cc_\Pi$ associated with a $(d+1)$-dimen\-sional deformed dg preprojective algebra $\Pi$.
\item[ii)] $\cc$ contains a $d$-cluster-tilting object $T$ and admits a dg enhancement $\cc_{dg}$ endowed with a right $d$-Calabi--Yau structure.
\end{itemize}
\end{thmx}

Then we give an example to illustrate the construction of the deformed dg preprojective algebra associated with a species with potential, \cf Example~\ref{ex:example}. In section~\ref{ss:proof from i) to ii) in main}, we prove the relatively easy implication from i) to ii) in Theorem~\ref{thm:B}. The existence of a (canonical) $d$-cluster-tilting object was already proved by
Amiot for $d=2$ and Guo for $d\geq 2$. For $d=1$, up to derived Morita equivalence, the $2$-dimensional dg preprojective algebra $\Pi$ is associated with a species $S$ of Dynkin type. Then the existence of a (canonical) $1$-cluster-tilting object follows from a triangle equivalence $\cd^b(A)/\tau \iso \cc_\Pi$, where $A$ is the finite-dimensional hereditary $k$-algebra associated with $S$. We construct the right $d$-Calabi--Yau structure on $\cc_{\Pi,dg}$ as follows: in a first step, from Van den Bergh's left $(d+1)$-Calabi--Yau structure on $\Pi$, we obtain a right $(d+1)$-Calabi--Yau structure on its Koszul dual, \cf Proposition~\ref{prop:left and right}. Using the dual connecting morphism associated with the defining sequence of $\cc_\Pi$ we obtain a candidate for a right $d$-Calabi--Yau structure on $\cc_{\Pi,dg}$. To prove that it is indeed non-degenerate, it suffices to do a check at the level of linear forms on Yoneda algebras in the triangulated category $\cc_\Pi$. Thanks to the description of the connecting morphism in Hochschild homology, \cf Theorem~\ref{thm:A}, this reduces to the fact that Amiot's construction yields a $d$-Calabi--Yau structure on the triangulated category $\cc_\Pi$, which was shown by Amiot and by Guo.

In section~\ref{ss:proof from ii) to i) in main}, we prove the relatively hard implication from ii) to i) in Theorem~\ref{thm:B}. In the proof, following Kalck--Yang \cite{KalckYang20}, we first use Palu's diagram \cite{Palu09a} to construct a complete augmented pseudo-compact dg algebra $\ol{\Ga}$ which is connective and smooth such that $\cc_{dg}$ becomes the dg quotient of $\per\!_{dg}\ol{\Ga}$ by $\pvd\!_{dg}\ol{\Ga}$. Our aim is to endow $\ol{\Ga}$ with a left $(d+1)$-Calabi--Yau
structure $[\tilde{\xi}]$ in order to apply Van den Bergh's theorem \cite{VandenBergh15} to $\ol{\Ga}$. The long exact sequence in cyclic homology combined with the passage to the Koszul dual using Proposition~\ref{prop:left and right} yields a candidate for $[\tilde{\xi}]$. In order to show that it is indeed non-degenerate, it suffices to show that the dual connecting morphism detects non-degeneracy, \cf Proposition~\ref{prop:detect non-degeneracy}. In the proof of this proposition, we compare the linear forms on $\pvd \ol{\Ga}$ with the ones in part~c) of Theorem~5.4 of \cite{KellerReiten07} which give rise to a $d$-Calabi--Yau structure on the triangulated category $\cc$. Thanks to Theorem~\ref{thm:A}, they coincide. Finally, we use Theorem~9.4 of Van den Bergh's \cite{VandenBergh15} to conclude that $\ol{\Ga}$ is indeed quasi-isomorphic to a $(d+1)$-dimensional deformed dg preprojective algebra. It follows that $\cc$ is triangle equivalent to the corresponding cluster category.

Section~\ref{section:relative} is devoted to a relative version of Theorem~\ref{thm:B}. In this version, the role of the cluster category $\cc_\Pi$ is played by the {\em Higgs category} $\ch_\gamma$ associated with a Ginzburg--Lazaroiu morphism $\gamma$ by Wu \cite{Wu23} and Keller--Wu \cite{KellerWu23}. In section~\ref{ss:assumptions}, we introduce the assumptions which are used in part~ii) of Theorem~\ref{thm:relative}. In Theorem~\ref{thm:relative}, we state the structure theorem in the relative case.

\begin{thmx}[=Theorem~\ref{thm:relative}] \label{thm:C}
Under suitable assumptions, for a Frobenius exact category $\ce$ with the full subcategory $\cp$ of projective-injective objects and a positive integer $d$, the following are equivalent.
\begin{itemize}
\item[i)] $\ce$ is equivalent, as an exact category, to the Higgs category $\ch_\gamma$ associated with a $(d+1)$-dimensional Ginzburg--Lazaroiu morphism $\gamma$.
\item[ii)] $\ce$ contains a $d$-cluster-tilting object $M$ and the inclusion functor $\ch^b(\cp) \to \cd^b(\ce)$ identifies with the inclusion of the image of $H^0(F)$ for a continuous dg functor $F\colon \tilde{\cp}\to \cd^b_{dg}(\ce)$ endowed with a left relative $(d+1)$-Calabi--Yau structure which yields a right $d$-Calabi--Yau structure on the dg quotient $\cc_{dg}=\cd^b_{dg}(\ce)/ \cd^b_{dg}(\cp)$.
\end{itemize}
\end{thmx}

We first prove some auxiliary lemmas: in Lemma~\ref{lem:exact sequence of Koszul duals}, we give a sufficient condition on an exact sequence of dg categories to yield an exact sequence linking their perfectly valued dg derived categories. In Lemma~\ref{lem:augmented}, we prove that the pseudo-compact derived endomorphism algebra of the cluster-tilting object $M$ is augmented and complete, which allows us to take the Koszul dual. In Proposition~\ref{prop:bijectivity}, we use this to show that under the assumptions of part~ii) in Theorem~\ref{thm:C}, the canonical map (\ref{eq:canonical map}) in relative negative cyclic homology is bijective.

In section~\ref{ss:proof from i) to ii) in relative}, we prove the relatively easy implication from i) to ii) in Theorem~\ref{thm:C}. The existence of a (canonical) $d$-cluster-tilting object in $\ce$ follows from that in the stable category of $\ce$. We define the pseudo-compact dg category $\tilde{\cp}$ to be $\per\!_{dg}\Pi_F$. From Keller--Liu's left relative $(d+1)$-Calabi--Yau structure on $\gamma\colon \Pi_F \to \Pi$, we obtain a left relative $(d+1)$-Calabi--Yau structure $[\tilde{\xi}]$ on the continuous dg functor $\per\!_{dg}\Pi_F \to \per\!_{dg}\Pi$ induced by $\gamma$. We construct the left relative $(d+1)$-Calabi--Yau structure on $F$ and the right \linebreak $d$-Calabi--Yau structure on $\cc_{dg}$ via $[\tilde{\xi}]$ as follows: for the first one, since the continuous dg functor $\per\!_{dg}\Pi \to \cd^b_{dg}(\ch_\gamma)$ is a pseudo-compact localisation, by the variant of part~(c) of Proposition~3.10 of \cite{Keller11b} for pseudo-compact dg categories, the class $[\tilde{\xi}]$ yields a left relative $(d+1)$-Calabi--Yau structure on $F$. For the second one, by the variant of Corollary~7.1 of \cite{BravDyckerhoff19} for pseudo-compact dg categories, the class $[\tilde{\xi}]$ yields a left $(d+1)$-Calabi--Yau structure on the homotopy cofibre of $\per\!_{dg}\Pi_F \to \per\!_{dg}\Pi$. Then by the proof of the implication from i) to ii) in Theorem~\ref{thm:B}, the left $(d+1)$-Calabi--Yau structure on the homotopy cofibre yields a right $d$-Calabi--Yau structure on the dg category $\cc_{dg}$.

In section~\ref{ss:proof from ii) to i) in relative}, we prove the relatively hard implication from ii) to i) in Theorem~\ref{thm:C}. In the proof, we use Palu's diagram \cite{Palu09a} to construct a continuous dg functor $\tilde{F}\colon \tilde{\cp}\to \per\!_{dg}\Ga$ such that $F$ becomes the composition of the pseudo-compact localisation $\per\!_{dg}\Ga \to \cd^b_{dg}(\ce)$ with $\tilde{F}$. Our aim is to endow $\tilde{F}$ with a left relative $(d+1)$-Calabi--Yau structure $[\tilde{\xi}]$ in order to apply Keller--Liu's structure theorem \cite{KellerLiu23} to $\tilde{F}$. The bijectivity of the canonical map (\ref{eq:canonical map}) yields a candidate for $[\tilde{\xi}]$. In order to show that it is indeed non-degenerate, we consider the induced morphism of pseudo-compact dg $\per\!_{dg}\Ga$-bimodules and show that its images under the derived pseudo-compact tensor functors with arbitrary right modules $L$ and left modules $M$ are isomorphisms. By the semi-orthogonal decomposition of $\cd(\per\!_{dg}\Ga)$, it reduces to two cases: a) the modules $L$ and $M$ are restricted from $\cd^b_{dg}(\ce)$ or b) the module $L$ or $M$ is induced from $\pvd\!_{dg}(\ol{\Ga})$. In case a), the claim follows because we have the given left relative $(d+1)$-Calabi--Yau structure on $F$. In case b), the induced module vanishes on $\tilde{\cp}$ so the claim reduces to the absolute case. In this case, it follows from the proof of Proposition~\ref{prop:detect non-degeneracy}. Finally, we use Theorem~4.3.1 of Keller--Liu's \cite{KellerLiu23} to conclude that $\Ga$ is indeed quasi-isomorphic to a $(d+1)$-dimensional deformed relative dg preprojective algebra. It follows that $\cd^b(\ce)$ is triangle equivalent to the corresponding relative cluster category. Upon inspection, we see that this equivalence induces an equivalence of exact categories between the Higgs category $\ch_\gamma$ and $\ce$.

\subsection*{Acknowledgements}

The authors thank Damien Calaque for pointing out
an error in a previous version of the crucial Proposition~\ref{prop:htpy short exact sequence of resolutions}.
The first-named author is grateful to Claire Amiot and Dong Yang for stimulating discussions and help with the references. 

The second-named author is supported by the Chinese Scholarship Council (CSC, Grant No.~202006210272). 
This article is part of his Ph.~D.~thesis defended at the Universit\'e Paris Cit\'e on December 15, 2023.

\section{Notations}

The following notations are used throughout the article: we let $k$ be a field. For a $k$-vector space $V$, we denote its $k$-dual space $\Hom_k(V,k)$ by $DV$. By abuse of notation, following \cite{VandenBergh15}, we write $a=a'\ten a''$ for an element $a=\sum_i a'_i\ten a''_i$ of a tensor product. Unless we specify otherwise, algebras have units but morphisms between algebras do not necessarily preserve the units. Modules are unital right modules. The internal degree of a homogeneous element $a$ in a graded vector space is denoted by $|a|$. We denote the shift functor of graded vector spaces by $\Si$ and write $\si \colon A\to \Si A$ for the canonical map of degree $-1$. We use cohomological grading so that differentials are of degree $1$. For any dg algebra, we denote its differential by $d$. We write $A^e$ for the enveloping algebra $A\ten_k A\op$ of any (dg) $k$-algebra $A$. For a graded $l$-bimodule $V$, we write $T_l V$ for the completed graded tensor algebra $\prod_{p\geq 0}V^{\ten_l p}$. We denote by $A^!$ the Koszul dual of a complete augmented pseudo-compact dg algebra $A$ in the sense of section~12.5 of \cite{VandenBergh15}.

\section{Preliminaries} \label{section:preliminaries}

\subsection{Pseudo-compact objects}

Following section~3 of \cite{VandenBergh15}, \cf also section~IV.3 of \cite{Gabriel62}, a {\em pseudo-compact vector space} is a topological vector space $V$ which has a basis of neighbourhoods of $0$ formed by distinguished subspaces of finite codimension such that $V$ is isomorphic to the inverse limit of the system formed by the quotients $V/V'$, where $V'$ runs through the open subspaces. A finite-dimensional vector space endowed with the discrete topology is a pseudo-compact vector space and conversely the topology on a finite-dimensional pseudo-compact vector space must be the discrete topology. Denote the category of pseudo-compact $k$-vector spaces by $\PC k$. Then we have a duality $\D \colon (\Mod k)\op \to \PC k$ which sends $V$ to its $k$-dual $DV=\Hom_k(V,k)$ endowed with the topology having a basis of neighbourhoods of $0$ formed by the kernels of the restriction maps $DV\to DV'$, where $V'$ runs through the finite-dimensional subspaces of $V$. Its quasi-inverse $\D \colon (\PC k)\op \to \Mod k$ sends $W$ to the $k$-vector space formed by the continuous $k$-linear maps from $W$ to $k$. The category $\PC k$ has a monoidal structure which is given by
\[
V\ten_k W=\D(\D W\ten_k \D V)
\]
for any $V$ and $W$ in $\PC k$. A {\em pseudo-compact graded vector space} is a graded vector space, where each component is endowed with a topology making it into a pseudo-compact vector space. The category of pseudo-compact graded $k$-vector spaces also has a monoidal structure as follows. For any pseudo-compact graded $k$-vector spaces $V$ and $W$, the component of degree $n$ of $V\ten_k W$ is given by
\[
\prod_{i+j=n}V_i\ten_k W_j \: .
\]
Following \cite{Gabriel62, VandenBergh01, KellerYang11}, a {\em pseudo-compact algebra} is a topological algebra $A$ which has a basis of neighbourhoods of $0$ formed by distinguished right ideals of finite codimension such that $A$ is isomorphic to the inverse limit of the system formed by the quotients $A/I$ as a right $A$-module, where $I$ runs through the open right ideals. Equivalently, it is a pseudo-compact vector space endowed with a continuous multiplication. Recall from part~(1) of Lemma~3.1 of \cite{VandenBergh15} that a pseudo-compact algebra also has a basis of neighbourhoods of $0$ formed by distinguished two-sided ideals of finite codimension. If we replace algebras by modules and right ideals by submodules, then we obtain the definition of {\em pseudo-compact modules}. If we replace algebras by graded algebras and right ideals by graded right ideals, then we obtain the definition of {\em pseudo-compact graded algebras}. A {\em pseudo-compact dg algebra} is a pseudo-compact graded algebra endowed with a continuous differential satisfying the graded Leibniz rule. Similarly, one can define {\em pseudo-compact dg modules}.

\subsection{Reminders and complements on pseudo-compact dg categories}

Let $A$ be a pseudo-compact dg $k$-algebra. We define $\cc(A)$ respectively $\ch(A)$ to be the category of pseudo-compact dg $A$-modules respectively the homotopy category and consider them as enriched over the category of $k$-vector spaces (not pseudo-compact $k$-vector spaces).
We write $\cd(A)$ for the (unbounded) derived category in the sense of section~5 of
\cite{VandenBergh15}. So the objects in $\cd(A)$ are the pseudo-compact dg $A$-modules
and its morphisms are obtained from morphisms of pseudo-compact dg $A$-modules by
localising with respect to a suitable class of weak equivalences (which is usually strictly contained in the class of quasi-isomorphisms). We consider $\cd(A)$ as enriched over the category of $k$-vector spaces (not pseudo-compact \linebreak $k$-vector spaces). 
Its thick subcategory generated by the free dg $A$-module of rank one is the {\em perfect derived category} $\per A$. Notice that usually, it does not consist of compact objects in $\cd(A)$.
We define the {\em perfectly valued derived category} $\pvd A$ to be the full subcategory of the 
{\em perfectly valued dg modules} in $\cd(A)$, \ie those dg modules $M$ whose homology is of finite total dimension. Clearly, an object $M$ of $\cd(A)$ belongs to $\pvd A$ if and only if the object $\RHom_A(A,M)$ belongs to $\per k=\pvd k$.
Denote the canonical dg enhancements of $\cd(A)$, $\per A$, $\pvd A$ by $\cd_{dg}(A)$, 
$\per\!_{dg}A$, $\pvd\!_{dg}A$ respectively.
Notice that one can also consider the above categories as enriched over the category of pseudo-compact $k$-vector spaces. We denote the pseudo-compact morphism space between objects $M$ and $N$ 
in $\cc(A)$ by $\Hom_A^{pc}(M,N)$. Its underlying vector space is $\Hom_A(M,N)$ and its topology is given by the basis of neighbourhoods of $0$ formed by the $\Hom_A(M,N')$, where $N'$ runs through the open dg $A$-submodules of $N$. Similarly for $\cd(A)$ and $\cd_{dg}(A)$.

Recall that $A$ is {\em connective} if its homology $H^p(A)$ vanishes for all $p>0$. 
In this case, the derived category $\cd(A)$ has a canonical t-structure whose aisles are
\begin{align*}
\cd(A)^{\leq 0} &= \{M\in \cd(A)\mid H^p(M)=0\mbox{ for all }p>0\}\mbox{ and}\\
\cd(A)^{\geq 0} &= \{M\in \cd(A)\mid H^p(M)=0\mbox{ for all }p<0\}\: .
\end{align*}
Its heart is equivalent to the module category of $H^0(A)$. The dg algebra $A$ is a {\em stalk algebra} if its homology $H^p(A)$ vanishes for all $p\ne 0$. In this case, we have the quasi-isomorphisms
\[
\begin{tikzcd}
A  & \tau_{\leq 0}(A) \arrow[swap]{l}{\sim} \arrow{r}{\sim} & H^0(A)
\end{tikzcd}
\]
so that $A$ is quasi-isomorphic to the ordinary algebra $H^0(A)$.
 
We write $A^e$ for the {\em enveloping algebra} $A\ten_k A\op$. Recall that $A$ is {\em smooth} if 
$A$ is perfect in $\cd(A^e)$ and that $A$ is {\em proper} if its underlying complex lies in $\per k$. Equivalently, $A$ is proper if and only if its homology $H^p(A)$ is finite-dimensional for all integers $p$ and vanishes for all $|p|\gg 0$. The dg algebra $A$ is {\em homologically finite} if its homology $H^p(A)$ is finite-dimensional for all integers $p$.

A {\em pseudo-compact category} is a category enriched over the monoidal category of pseudo-compact vector spaces. It is {\em (Frobenius) exact} if its underlying category is endowed with a (Frobenius) exact structure.
A {\em pseudo-compact dg category} is a category enriched over the monoidal category of
complexes of pseudo-compact vector spaces. It is {\em pretriangulated} if its underlying
dg category is pretriangulated. A {\em continuous dg functor} between pseudo-compact dg categories is a dg functor which induces continuous maps between morphism spaces.

\subsection{Augmented algebras and morphisms} \label{ss:morphisms between pc algebras}

Suppose that $l$ is a finite-dimensional semisimple $k$-algebra. An {\em $l$-augmented dg algebra} is a dg $k$-algebra $A$ endowed with dg $k$-algebra morphisms
\[
\begin{tikzcd}
l \arrow{r}{\eta} & A \arrow{r}{\eps} & l
\end{tikzcd}
\]
satisfying $\eps \circ \eta =\id_{l}$, \cf section~5 of \cite{VandenBergh15}. We call $\eta$ the unit map and $\eps$ the augmentation map. They yield an $l$-bimodule isomorphism $A\simeq l\oplus \ol{A}$, where $\ol{A}$ denotes the the kernel of $\eps$. Similarly for $l$-augmented dg coalgebras. Suppose that $l_A$ and $l_B$ are finite-dimensional semisimple $k$-algebras and $\phi \colon l_B \to l_A$ is a morphism between $k$-algebras (not necessarily preserving the unit!). A {\em $\phi$-augmented morphism between dg algebras} is a morphism \linebreak $f\colon B \to A$ between dg $k$-algebras (not necessarily preserving the unit!) fitting into a commutative diagram
\[
\begin{tikzcd}
l_B \arrow{r}{\eta_B} \arrow[swap]{d}{\phi} & B \arrow{r}{\eps_B} \arrow{d}{f} & l_B \arrow{d}{\phi} \\
l_A \arrow[swap]{r}{\eta_A} & A \arrow[swap]{r}{\eps_A} & l_A\mathrlap{\: .}
\end{tikzcd}
\]
If the above dg algebras are pseudo-compact and morphisms are continuous, we obtain the definition of {\em augmented pseudo-compact dg algebras} and {\em augmented morphisms between pseudo-compact dg algebras}.
For a pseudo-compact dg algebra $A$, its {\em radical} $\rad A$ is defined to be the common annihilator of all the simple pseudo-compact dg $A$-modules. An \linebreak $l$-augmented pseudo-compact dg algebra $A$ is {\em complete} if $\ol{A}$ equals the radical $\rad A$. This holds if and only if the augmented dg coalgebra $C=\D A$ is {\em cocomplete}, \ie $\ol{C}$ is the union of the kernels of the iterated comultiplications $\ol{C}\to \ol{C}^{\ten_l n}$, $n\geq 2$. We endow the category of cocomplete $l$-augmented dg coalgebras with the model category structure described in section~1.3.1 of \cite{Lefevre03}, \cf also section~12.1 of \cite{VandenBergh15}.
Following section~12.4 of \cite{VandenBergh15}, we endow the category of complete $l$-augmented pseudo-compact dg algebras with the dual model category structure. We endow the category of cocomplete dg 
comodules over a cocomplete $l$-augmented dg coalgebra with the model structure of 
section~1.3.1 of \cite{Lefevre03}, \cf also section~12.2 of \cite{VandenBergh15}. Following
section~12.4 of \cite{VandenBergh15}, we endow the category of pseudo-compact dg modules over a complete $l$-augmented pseudo-compact dg algebra with the dual model category structure.

\subsection{Reminders on extriangulated categories}

Following section~2 of \cite{NakaokaPalu19}, the notion of extriangulated category generalises both the notion of exact category and that of triangulated category. It arises naturally in the categorification of cluster algebras with coefficients, \cf \cite{Wu23, Wu21, KellerWu23}.

Let $\ce$ be an additive category and $\mathbb{E}\colon \ce\op \times \ce \to \Mod\Z$ a biadditive bifunctor. We call the elements $\delta$ of $\mathbb{E}(Z,X)$ {\em extensions} of $Z$ by $X$. A {\em morphism of extensions} from $\delta \in \mathbb{E}(Z,X)$ to $\delta' \in \mathbb{E}(Z',X')$ is a pair $(u,w)\in \ce(X,X')\times \ce(Z,Z')$ satisfying $u_*(\delta)=w^*(\delta')$. Two sequences of morphisms
\[
\begin{tikzcd}
X\arrow{r}{i} & Y\arrow{r}{p} & Z
\end{tikzcd}
\quad \mbox{and} \quad
\begin{tikzcd}
X\arrow{r}{i'} & Y'\arrow{r}{p'} & Z
\end{tikzcd}
\]
are {\em equivalent} if there is a morphism $v\colon Y\to Y'$ such that the diagram
\[
\begin{tikzcd}
X\arrow{r}{i}\arrow[equal]{d} & Y\arrow{r}{p}\arrow{d}{v} & Z\arrow[equal]{d} \\
X\arrow[swap]{r}{i'} & Y'\arrow[swap]{r}{p'} & Z
\end{tikzcd}
\]
commutes. A {\em realisation} for $\mathbb{E}$ is a family $s$ of maps
\[
\begin{tikzcd}
s_{Z,X}\colon \mathbb{E}(Z,X)\arrow{r} & \{\mbox{equivalence class of sequences }X\xlongrightarrow{i} Y\xlongrightarrow{p} Z\}\ko X, Z\in \ce \: .
\end{tikzcd}
\]
The sequences
\[
\begin{tikzcd}
X\arrow{r}{i} & Y\arrow{r}{p} & Z
\end{tikzcd}
\]
in the image of $s$ are called {\em extriangles} or {\em conflations}, the morphism $i$ {\em inflations} and the morphism $p$ {\em deflations}. We write
\[
\begin{tikzcd}
X\arrow{r}{i} & Y\arrow{r}{p} & Z\arrow[dashed]{r}{\delta} & \mbox{}
\end{tikzcd}
\]
to indicate that $s(\delta)=(i,p)$. A realisation $s$ is {\em additive} if
\begin{itemize}
\item[a)] for each morphism of extensions $(u,w)\colon \delta \to \delta'$, there is a morphism $v$ fitting into the commutative diagram
\[
\begin{tikzcd}
X\arrow{r}\arrow{d}{u} & Y\arrow{r}\arrow{d}{v} & Z\arrow[dashed]{r}{\delta}\arrow{d}{w} & \mbox{} \\
X'\arrow{r} & Y'\arrow{r} & Z'\arrow[dashed]{r}{\delta'} & \mathrlap{\: ,}
\end{tikzcd}
\]
\item[b)] $s_{Z,X}(0)$ is a split extriangle
\[
\begin{tikzcd}[ampersand replacement=\&]
X\arrow{r}{
\begin{bsmallmatrix}
\id_X \\
0
\end{bsmallmatrix}
} \& X\oplus Z\arrow{r}{
\begin{bsmallmatrix}
0 & \id_Z
\end{bsmallmatrix}
} \& Z\arrow[dashed]{r}{0} \& \mathrlap{\: ,}
\end{tikzcd}
\]
\item[c)] for $\delta \in \mathbb{E}(Z,X)$ and $\delta' \in \mathbb{E}(Z',X')$, the realisation of $
\begin{bmatrix}
\delta & 0 \\
0 & \delta'
\end{bmatrix}
\in \mathbb{E}(Z\oplus Z',X\oplus X')$ is the direct sum of the extriangles
\[
\begin{tikzcd}
X\arrow{r} & Y\arrow{r} & Z\arrow[dashed]{r}{\delta} & \mbox{}
\end{tikzcd}
\quad \mbox{and} \quad
\begin{tikzcd}
X'\arrow{r} & Y'\arrow{r} & Z'\arrow[dashed]{r}{\delta'} & \mathrlap{\: .}
\end{tikzcd}
\]
\end{itemize}
The triple $(u,v,w)$ in part~a) is called a {\em morphism of extriangles}.
An {\em extriangulated category} is a triple $(\ce,\mathbb{E},s)$, where
\begin{itemize}
\item[ET0)\phantom{$\op$}] $\ce$ is an additive category,
\item[ET1)\phantom{$\op$}] $\mathbb{E}\colon \ce\op \times \ce \to \Mod\Z$ is a biadditive bifunctor,
\item[ET2)\phantom{$\op$}] s is an additive realisation,
\item[ET3)\phantom{$\op$}] given extriangles
\[
\begin{tikzcd}
X\arrow{r}{i} & Y\arrow{r}{p} & Z\arrow[dashed]{r}{\delta} & \mbox{}
\end{tikzcd}
\ko
\begin{tikzcd}
X'\arrow{r}{i'} & Y'\arrow{r}{p'} & Z'\arrow[dashed]{r}{\delta'} & \mbox{}
\end{tikzcd}
\]
and a commutative square
\[
\begin{tikzcd}
X\arrow{r}{i}\arrow[swap]{d}{u} & Y\arrow{d}{v} \\
X'\arrow[swap]{r}{i'} & Y'\mathrlap{\: ,}
\end{tikzcd}
\]
there is a morphism $(u,w)\colon \delta \to \delta'$ such that the diagram
\[
\begin{tikzcd}
X\arrow{r}{i}\arrow{d}{u} & Y\arrow{r}{p}\arrow{d}{v} & Z\arrow{d}{w} \\
X'\arrow[swap]{r}{i'} & Y'\arrow[swap]{r}{p'} & Z'
\end{tikzcd}
\]
commutes,
\item[ET3)$\op$] dual of ET3),
\item[ET4)\phantom{$\op$}] given extriangles
\[
\begin{tikzcd}
X\arrow{r}{i} & Y\arrow{r}{p} & Z\arrow[dashed]{r}{\delta} & \mbox{}
\end{tikzcd}
\quad \mbox{and} \quad
\begin{tikzcd}
Y\arrow{r}{j} & T\arrow{r}{q} & U\arrow[dashed]{r}{\delta'} & \mathrlap{\: ,}
\end{tikzcd}
\]
we can fit them into a diagram
\[
\begin{tikzcd}
X\arrow{r}{i}\arrow[equal]{d} & Y\arrow{r}{p}\arrow{d}{j} & Z\arrow[dashed]{r}{\delta}\arrow{d}{k} & \mbox{} \\
X\arrow{r}{j\circ i} & T\arrow{r}{t}\arrow{d}{q} & Y'\arrow[dashed]{r}{\eps}\arrow{d}{r} & \mbox{} \\
 & U\arrow[equal]{r}\arrow[dashed]{d}{\delta'} & U\arrow[dashed]{d}{p_* \delta'} & \\
 & \mbox{} & \mbox{} & 
\end{tikzcd}
\]
such that $(\id_X,j,k)$ and $(p,t,\id_U)$ are morphisms of extriangles and $(i,r)\colon \eps \to \delta'$ is a morphism of extensions,
\item[ET4)$\op$] dual of ET4).
\end{itemize}

\begin{example}
\begin{itemize}
\item[a)] If $\ce$ is exact, take $\mathbb{E}=\Ext^1_\ce(?,-)$ and $s$ the classical realisation taking each $\delta \in \Ext^1_\ce(Z,X)$ to the corresponding equivalence class of conflation
\[
\begin{tikzcd}
X\arrow[rightarrowtail]{r}{i} & Y\arrow[twoheadrightarrow]{r}{p} & Z\: .
\end{tikzcd}
\]
Then the triple $(\ce,\mathbb{E},s)$ is an extriangulated category.
A morphism $f\colon X\to X'$ yields
\[
\begin{tikzcd}
f_*=\Ext^1(Z,f)\colon \Ext^1(Z,X)\arrow{r} & \Ext^1(Z,X')
\end{tikzcd}
\]
taking $\delta$ to $f_*(\delta)$ and a morphism $g\colon Z'\to Z$ yields
\[
\begin{tikzcd}
g^*=\Ext^1(g,X)\colon \Ext^1(Z,X)\arrow{r} & \Ext^1(Z',X)
\end{tikzcd}
\]
taking $\delta$ to $g^*(\delta)$. These correspond to the commutative diagram
\[
\begin{tikzcd}
X\arrow[rightarrowtail]{r}\arrow[equal]{d} & Y'\arrow[twoheadrightarrow]{r}\arrow{d}\arrow[phantom, very near start]{dr}{\lrcorner} & Z'\arrow[dashed]{r}{g^*(\delta)}\arrow{d}{g} & \mbox{} \\
X\arrow[rightarrowtail]{r}{i}\arrow[swap]{d}{f} & Y\arrow[twoheadrightarrow]{r}{p}\arrow{d} & Z\arrow[dashed]{r}{\delta}\arrow[equal]{d} & \mbox{} \\
X'\arrow[rightarrowtail]{r} & Y''\arrow[twoheadrightarrow]{r}\arrow[phantom, very near start]{ul}{\ulcorner} & Z\arrow[dashed]{r}{f_*(\delta)} & \mathrlap{\: .}
\end{tikzcd}
\]
For a morphism of extensions $(u,w)\colon \delta \to \delta'$, the equality $u_*(\delta)=w^*(\delta')$ means that there is a morphism $v\colon Y\to Y'$ fitting into the commutative diagram
\[
\begin{tikzcd}
X\arrow[rightarrowtail]{r}{i}\arrow{d}{u} & Y\arrow[twoheadrightarrow]{r}{p}\arrow{d}{v} & Z\arrow[dashed]{r}{\delta}\arrow{d}{w} & \mbox{} \\
X'\arrow[rightarrowtail,swap]{r}{i'} & Y'\arrow[twoheadrightarrow,swap]{r}{p'} & Z'\arrow[dashed]{r}{\delta'} & \mathrlap{\: .}
\end{tikzcd}
\]
\item[b)] If $\ce$ is triangulated, take $\mathbb{E}=\ce(?,\Si -)$ and $s$ the classical realisation taking each $\delta \in \ce(Z,\Si X)$ to the corresponding equivalence class of triangle
\[
\begin{tikzcd}
X\arrow{r} & Y\arrow{r} & Z\arrow{r}{\delta} & \Si X\: ,
\end{tikzcd}
\]
then the triple $(\ce,\mathbb{E},s)$ is an extriangulated category. A morphism $f\colon X\to X'$ yields $f_*=\ce(Z,\Si f)$ and a morphism $g\colon Z'\to Z$ yields $g^*=\ce(g,\Si X)$ and these correspond to a similar diagram as above.
For a morphism of extensions $(u,w)\colon \delta \to \delta'$, the equality $u_*(\delta)=w^*(\delta')$ means that the square
\[
\begin{tikzcd}
Z\arrow{r}{\delta}\arrow[swap]{d}{w} & \Si X\arrow{d}{\Si u} \\
Z'\arrow[swap]{r}{\delta'} & \Si X'
\end{tikzcd}
\]
commutes.
\item[c)] Let $\ct$ be a triangulated category and $\ce \subseteq \ct$ an extension closed full subcategory. Then $\ce$ with $\mathbb{E}=\ct(?,\Si -)$ and the induced realisation is an extriangulated category.
\end{itemize}
\end{example}

\section{On the virtues of pseudo-compact dg algebras} \label{section:pseudo-compact dg algebras}

\subsection{Construction of augmentations}

Suppose that $l$ is a finite-dimensional semisimple $k$-algebra.

\begin{proposition} \label{prop:complete augmented}
Let $A$ be a connective pseudo-compact dg algebra such that $H^0(A)$ is a complete $l$-augmented pseudo-compact algebra. If the algebra $l$ is separable, then $A$ is quasi-isomorphic to a complete $l$-augmented pseudo-compact dg algebra concentrated in non-positive degrees.
\end{proposition}

\begin{remark} \label{rk:Wedderburn}
If $B$ is a pseudo-compact algebra whose radical quotient
$l=B/\rad B$ is finite-dimensional and separable, then by the Wedderburn--Mal'cev theorem, the algebra morphism $B \to l$ admits a section and $B$ becomes a complete $l$-augmented pseudo-compact algebra. So if $A$ is a connective pseudo-compact dg algebra such that the radical quotient $l$ of $H^0(A)$ is finite-dimensional and separable, then $A$ is quasi-isomorphic to a complete 
$l$-augmented pseudo-compact dg algebra concentrated in non-positive degrees.
\end{remark}

\begin{proof}[Proof of Proposition~\ref{prop:complete augmented}]
By truncating we may and will assume that $A$ is concentrated in non-positive degrees. We have the surjective morphisms
\[
\begin{tikzcd}
A^0 \arrow{r} & H^0(A) \arrow{r} & l
\end{tikzcd}
\]
of pseudo-compact algebras which induce a surjective morphism $A^0/\rad(A^0) \to l$ of semisimple pseudo-compact algebras. This morphism has an algebra section $s$ (not necessarily preserving the unit). Then $s(1_l)$ is an idempotent of $A^0/\rad(A^0)$. Since the algebra $A^0$ is pseudo-compact, by Proposition~III.2.10 of \cite{Bass68}, the idempotent $s(1_l)$ lifts to an idempotent $e$ of $A^0$. By construction, the element $1_l-e$ is a boundary. So the inclusion $eA\to A$ is a quasi-isomorphism of pseudo-compact dg $A$-modules. Therefore, the pseudo-compact dg algebra $eAe=\RHom_A(eA,eA)$ 
is quasi-isomorphic to $A=\RHom_A(A,A)$. Since we have isomorphisms
\[
\begin{tikzcd}
eAe/\rad(eAe)\arrow{r}{\sim} & eA^0 e/\rad(eA^0 e)\arrow{r}{\sim} & s(1_l)(A^0/\rad(A^0))s(1_l)\arrow{r}{\sim} & l
\end{tikzcd}
\]
of pseudo-compact dg algebras and $l$ is separable, by the Wedderburn--Mal'cev theorem, the dg algebra $eAe$ underlies a complete $l$-augmented pseudo-compact dg algebra (concentrated in non-positive degrees).
\end{proof}

\subsection{Minimality, perfection and smoothness in the pseudo-compact setting}

\begin{proposition} \label{prop:minimal cofibrant replacement}
Let $A$ be a complete $l$-augmented pseudo-compact dg algebra and $M$ a pseudo-compact dg $A$-module. Then $M$ is weakly equivalent to a cofibrant pseudo-compact dg $A$-module $P$
which is {\em{minimal}}, \ie we have $d(P)\subseteq P\cdot \rad{A}$.
\end{proposition}
\begin{proof}
The coalgebra $C=\D A$ is a cocomplete $l$-augmented coalgebra. By page~41 of \cite{Lefevre03}, it is weakly equivalent to $B_\infty \La$ for a minimal $l$-augmented $A_\infty$-algebra $\La$. Without loss of generality, we may and will assume that the coalgebra $C$ equals $B_\infty \La$. By Proposition~2.4.1.3 of \cite{Lefevre03}, the cocomplete dg $C\op$-comodule $\D M$ is weakly equivalent to a direct summand of $B_\infty N$ for a minimal $A_\infty$-module $N$ over $\La\op$. Since the class of dg modules $M$ satisfying the claim is stable under passage to direct summands, we may and will assume that $\D M$ is weakly equivalent to $B_\infty N$. Put $P=\D B_\infty N$. It is weakly equivalent to $M$. Since $B_\infty N$ is isomorphic to $C\ten_l N$ as a graded $C\op$-comodule, we have
\[
P\simeq \D(C\ten_l N)=\D N\ten_l \D C=\D N\ten_l A
\]
as pseudo-compact graded $A$-modules. So the pseudo-compact dg $A$-module $P$ is cofibrant. Since $N$ is minimal, the composed map
\[
\begin{tikzcd}
l\ten_l N\arrow{r}{\eta \ten \id_N} & B_\infty N\arrow{r}{d} & B_\infty N\arrow{r}{\eps \ten \id_N} & l\ten_l N
\end{tikzcd}
\]
is the zero map. So the dual map
\[
\begin{tikzcd}
\D N\ten_l l\arrow{r}{\id_{\D N}\ten \eta} & P\arrow{r}{d} & P\arrow{r}{\id_{\D N}\ten \eps} & \D N\ten_l l
\end{tikzcd}
\]
is also the zero map. Since the image of the map $\id_{\D N}\ten \eta$ generates $P$ as a pseudo-compact dg $A$-module, the image of $d$ is contained in the kernel of $\id_{\D N}\ten \eps$, which equals $P\cdot \rad A$.
\end{proof}

\begin{corollary} \label{cor:perfect}
Let $A$ be a complete $l$-augmented pseudo-compact dg algebra and $M$ a pseudo-compact dg $A$-module. Then $M$ is perfect if and only if $\RHom_A(M,l)$ lies in $\per k$.
\end{corollary}
\begin{proof}
The necessity of the condition is obvious. Let us prove that it is sufficient.
By Proposition~\ref{prop:minimal cofibrant replacement}, the pseudo-compact dg $A$-module $M$ is weakly equivalent to a cofibrant pseudo-compact dg $A$-module $P\simeq V\ten_l A$ for a pseudo-compact graded $l$-module $V$ satisfying $d(P)\subseteq P\cdot \rad{A}$. So the complex
\[
\begin{tikzcd}
\RHom_A(M,l)\arrow[no head]{r}{\sim} & \cc_{dg}(A)(V\ten_l A,l)\arrow{r}{\sim} & \Hom_l(V,l)\arrow[no head]{r}{\sim} & \Hom_k(V,k)
\end{tikzcd}
\]
has vanishing differential, where the last isomorphism follows from Lemma~3.2.3 of \cite{KellerLiu23}. Since it lies in $\per k$, the pseudo-compact graded $l$-module $V$ is of finite total dimension. Therefore, the pseudo-compact dg $A$-module $P$ is perfect and hence so is $M$.
\end{proof}

\begin{corollary} \label{cor:smooth}
Let $A$ be a complete $l$-augmented pseudo-compact dg algebra. Then the following are equivalent.
\begin{itemize}
\item[i)] $A$ is smooth.
\item[ii)] $\pvd A$ is contained in $\per A$.
\item[iii)] $A^!$ is proper.
\end{itemize}
\end{corollary}

\begin{remark} The implications from i) to ii) and from ii) to iii) are well-known and also hold
for dg algebras $A$ which are not pseudo-compact. The implication from iii) to i) needs the pseudo-compactness assumption.
\end{remark}

\begin{proof}[Proof of Corollary~\ref{cor:smooth}]
Let us show that i) implies ii): let $M$ be an object in $\pvd A$. Then the underlying complex of $M$ lies in $\per k$ and therefore, the object $M\lten_k A$ lies
in $\per A$. Since $A$ is perfect over $A^e$, the object $M \liso M \lten_A A$ lies in the thick subcategory of $\cd(A)$ generated by $M\lten_A A^e \iso M\lten_k A$ and so in $\per A$.

Clearly ii) implies iii): since the object $\RHom_A(A,l)$ lies in $\per k$ and $l$ lies in $\per A$, the underlying complex of $A^!\simeq \RHom_A(l,l)\op$ lies in $\per k$.

We show that iii) implies i): we have isomorphisms
\[
\begin{tikzcd}
\RHom_{A^e}(A,l^e)\arrow{r}{\sim} & \RHom_{A^e}(A,\Hom_k(l,l)) & \RHom_{A\op}(l,l)\arrow[swap]{l}{\sim} \: .
\end{tikzcd}
\]
Since the underlying complex of $\RHom_{A\op}(l,l)\simeq A^!$ lies in $\per k$, by Corollary~\ref{cor:perfect}, the pseudo-compact dg $A^e$-module $A$ is perfect.
\end{proof}

\begin{lemma} \label{lem:forgetful functor}
Let $A$ be a smooth pseudo-compact dg algebra and $A^u$ its underlying dg algebra. Then the forgetful functor $\pvd A \to \pvd A^u$ is fully faithful.
\end{lemma}
\begin{proof}
The forgetful functor from $\per A \to \per A^u$ is an equivalence because both categories are generated by $A$ and the homology of $A^u$ coincides with the homology of $A$. Since the pseudo-compact dg algebra $A$ is smooth, by the implication from i) to ii) in Corollary~\ref{cor:smooth}, the equivalence $\per A \to \per A^u$ restricts to a fully faithful functor $\pvd A \to \pvd A^u$.
\end{proof}

\begin{lemma} \label{lem:pseudo-compact localisation}
Let $A$ and $B$ be pseudo-compact dg algebras and $A^u$ respectively $B^u$ their underlying dg algebras. Let $f\colon A\to B$ be a continuous morphism between pseudo-compact dg algebras and $f^u$ its underlying morphism between dg algebras. If $A$ is smooth and the induction functor $f^{u*}\colon \cd(A^u)\to \cd(B^u)$ is a localisation, then $f$ is a pseudo-compact localisation, \ie the multiplication morphism $\mu\colon B\lten_A B \to B$ in $\cd(B^e)$ is an isomorphism.
\end{lemma}
\begin{proof}
Since the category $\cd(B^e)\op$ is compactly generated by $(\pvd(B^e))\op$, it suffices to show that the image of $\mu$ under the functor $\RHom_{B^e}(?,\Hom_k(L,M))$ is an isomorphism for all $L$ and $M$ in $\pvd B$. Since there are isomorphisms
\[
\begin{tikzcd}
\RHom_{B^e}(B,\Hom_k(L,M)) & \RHom_B(L,M)\arrow[swap]{l}{\sim}
\end{tikzcd}
\]
and
\[
\begin{tikzcd}
\RHom_{B^e}(B\lten_A B,\Hom_k(L,M))\arrow[no head]{r}{\sim} & \RHom_{A^e}(A,\Hom_k(L,M)) & \RHom_A(L,M)\arrow[swap]{l}{\sim}\: ,
\end{tikzcd}
\]
it suffices to show that the morphism
\begin{equation} \label{eq:restriction}
\begin{tikzcd}
\RHom_B(L,M)\arrow{r} & \RHom_A(L,M)
\end{tikzcd}
\end{equation}
induced by the restriction functor $f_*\colon \pvd B\to \pvd A$ is an isomorphism. Since the induction functor $f^{u*}\colon \cd(A^u)\to \cd(B^u)$ is a localisation, the restriction functor
\[
f^u_*\colon \pvd B^u\to \pvd A^u
\]
is fully faithful. By Lemma~\ref{lem:forgetful functor}, the vertical functors in the commutative square
\[
\begin{tikzcd} 
\pvd B \arrow{r}{f_*} \arrow{d} & \pvd A \arrow{d} \\
\pvd B^u \arrow[swap]{r}{f_*^u} & \pvd A^u
\end{tikzcd}
\]
are fully faithful. Thus, the functor $f_* \colon \pvd B \to \pvd A$ is fully faithful and the morphism (\ref{eq:restriction}) is an isomorphism.
\end{proof}

\subsection{Reminders on Calabi--Yau structures} \label{ss:reminders on Calabi--Yau structures}

In this section, we recall the necessary background on Hochschild and cyclic homology, absolute and relative Calabi--Yau structures. We work in the setting of (pseudo-compact) dg algebras but everything generalises to the setting of small (pseudo-compact) dg categories.

Following section~1 of \cite{Kassel85}, a {\em mixed complex} over $k$ is a dg module over the dg algebra $\La=k[t]/(t^2)$, where $t$ is an indeterminate of degree $-1$ satisfying $d(t)=0$. Let $l$ be a finite-dimensional separable $k$-algebra. For a dg $l$-algebra $A$, its {\em bar resolution} $B(A)$ is defined to be the sum total complex of
\[
\begin{tikzcd}
  \cdots\arrow{r} & A^{\ten_l 4} \arrow{r}{b'} & A^{\ten_l 3} \arrow{r}{b'} & A^{\ten_l 2}\: ,
\end{tikzcd}
\]
where the differential of $A^{\ten_l p}$ sends $a_1\ten\cdots\ten a_p$ to
\[
\sum_{i=1}^p (-1)^{i-1+|a_1|+\cdots+|a_{i-1}|}a_1\ten\cdots\ten d(a_i)\ten\cdots\ten a_p
\]
and the differential $b'$ sends $a_1\ten\cdots\ten a_p$ to 
\[
\sum_{i=1}^{p-1}(-1)^{i-1+|a_1|+\cdots+|a_i|}a_1\ten\cdots\ten a_i a_{i+1}\ten\cdots\ten a_p\: .
\]
Its {\em mixed complex} $M(A)$ is defined as follows. Its underlying complex is defined to be the cone of the map $\id-\tau$ from the sum total complex $B^+(A)_l$ of
\[
\begin{tikzcd}
  \cdots\arrow{r} & (A^{\ten_l 3})_l \arrow{r}{b'} & (A^{\ten_l 2})_l \arrow{r}{b'} & A_l
\end{tikzcd}
\]
to the sum total complex $C(A)$ of
\[
\begin{tikzcd}
  \cdots\arrow{r} & (A^{\ten_l 3})_l \arrow{r}{b} & (A^{\ten_l 2})_l \arrow{r}{b} & A_l \: .
\end{tikzcd}
\]
Here $\tau$ sends $a_1\ten\cdots\ten a_p$ to
\[
(-1)^{(|a_p|+1)(p-1+|a_1|+\cdots+|a_{p-1}|)}a_p\ten a_1\ten \cdots\ten a_{p-1}\: ,
\]
the differential of $(A^{\ten_l p})_l$ is given by the same formula as above, the map $b$ is the differential of the Hochschild chain complex and $b'$ is induced by that of the augmented bar resolution. Explicitly, the differential $b$ sends $a_1\ten\cdots\ten a_p$ to
\begin{align*}
& \sum_{i=1}^{p-1}(-1)^{i-1+|a_1|+\cdots+|a_i|} a_1\ten\cdots\ten a_i a_{i+1}\ten\cdots\ten a_p \\
& +(-1)^{(|a_p|+1)(p+|a_1|+\cdots+|a_{p-1}|)-1}a_p a_1\ten\cdots\ten a_{p-1}
\end{align*}
and $b'$ is given by the same formula as above. The $\La$-module structure on $M(A)$ is determined by the action of $t$, which vanishes on $B^+(A)_l$ and sends the component $(A^{\ten_l p})_l$ of $C(A)$ to the corresponding component of $B^+(A)_l$ via the map $\sum_{i=0}^{p-1}\tau^i$.

The {\em Hochschild complex} $HH(A)$ of $A$ is defined to be the underlying complex of $M(A)$. 
Via the canonical projection from the underlying complex of $M(A)$ to $C(A)$, it is quasi-isomorphic to $C(A)\liso A\lten_{A^e}A$ because each row of $B^+(A)_l$ is acyclic.
The {\em cyclic complex} $HC(A)$ of $A$ is defined to be the complex $M(A)\lten_\La k$. The {\em negative cyclic complex} $HN(A)$ of $A$ is defined to be the complex $\RHom_\La (k,M(A))$. The {\em periodic cyclic complex} $HP(A)$ of $A$ is defined to be the derived inverse limit of the system
\[
\begin{tikzcd}
\cdots \arrow{r} & \Si^{-4}\RHom_\La (k,M(A))\arrow{r} & \Si^{-2}\RHom_\La (k,M(A))\arrow{r} & \RHom_\La (k,M(A))\: ,
\end{tikzcd}
\]
where the transition maps are induced by the canonical map $k \to \Si^2 k$.
For a dg $A$-bimodule $M$, the {\em Hochschild complex of $A$ with coefficients in $M$} denoted by $HH(A,M)$ is defined to be the sum total complex of
\[
\begin{tikzcd}
  \cdots\arrow{r} & (M\ten_l A^{\ten_l 2})_l \arrow{r}{b} & (M\ten_l A)_l \arrow{r}{b} & M_l \: ,
\end{tikzcd}
\]
where the differential of $(M\ten_l A^{\ten_l p})_l$ sends $m\ten a_1\ten\cdots\ten a_p$ to
\[
d(m)\ten a_1 \ten\cdots\ten a_p +\sum_{i=1}^p (-1)^{i+|m|+|a_1|+\cdots+|a_{i-1}|}m\ten a_1 \ten\cdots\ten d(a_i)\ten\cdots\ten a_p
\]
and the differential $b$ sends $m\ten a_1\ten\cdots\ten a_p$ to
\begin{align*}
& (-1)^{|m|}ma_1\ten\cdots\ten a_p \\
& +\sum_{i=1}^{p-1}(-1)^{i+|m|+|a_1|+\cdots+|a_i|} m\ten a_1\ten\cdots\ten a_i a_{i+1}\ten\cdots\ten a_p \\
& +(-1)^{(|a_p|+1)(p+1+|m|+|a_1|+\cdots+|a_{p-1}|)-1}a_p m\ten a_1\ten\cdots\ten a_{p-1}\: .
\end{align*}
Their homologies are called {\em Hochschild homology} $HH_*(A)$, {\em cyclic homology} $HC_*(A)$, {\em negative cyclic homology} $HN_*(A)$, {\em periodic cyclic homology} $HP_*(A)$, {\em Hochschild homology with coefficients in $M$} denoted by $HH_*(A,M)$ respectively. Because $l$ is separable, the \linebreak $l$-coinvariant functor is exact. This implies that the mixed complexes of $A$ and the Hochschild complexes of $A$ with coefficients in $M$, where $A$ and $M$ are considered as graded $l$-bimodules and considered as graded vector spaces, are quasi-isomorphic. We refer the reader to section~6.1 of \cite{VandenBergh15} for the analogues of the above constructions for pseudo-compact dg algebras (where sum total complexes have to be replaced with product total complexes). In the sequel, we use them freely.

Fix an integer $d$. For a smooth pseudo-compact dg $l$-algebra $A$, a {\em left (absolute) \linebreak $d$-Calabi--Yau structure} on $A$ is a class $[\tilde{\xi}]$ in $HN_d(A)$ which is {\em non-degenerate}, \ie the morphism $\Si^d A^\vee \to A$ in $\cd(A^e)$ obtained from the image $[\xi]$ of $[\tilde{\xi}]$ under the canonical map $HN_d(A)\to HH_d(A)$ via
\[
\begin{tikzcd}
HH_d(A)\arrow[no head]{r}{\sim} & H^{-d}(A\lten_{A^e}A)\arrow{r} & H^{-d}(\RHom_{A^e}^{pc}(A^\vee,A))\arrow[no head]{r}{\sim} & \Hom_{\cd(A^e)}^{pc}(\Si^d A^\vee,A)
\end{tikzcd}
\]
is an isomorphism, where we denote the derived bimodule dual $\RHom_{A^e}^{pc}(A,A^e)$ of $A$ by $A^\vee$.

For a homologically finite dg $l$-algebra $A$ (not supposed to be pseudo-compact), a {\em right (absolute) $d$-Calabi--Yau structure} on $A$ is a class $[\tilde{x}]$ in $DHC_{-d}(A)$ which is {\em non-degenerate}, \ie the morphism $A\to \Si^{-d}DA$ in $\cd(A^e)$ obtained from the image $[x]$ of $[\tilde{x}]$ under the canonical map $DHC_{-d}(A)\to DHH_{-d}(A)$ via 
\[
DHH_{-d}(A)\simeq H^{-d}(D(A\lten_{A^e}A))\simeq H^{-d}(\RHom_{A^e}(A, DA))\simeq \Hom_{\cd(A^e)}(A,\Si^{-d}DA)
\]
is an isomorphism.

Recall that a $k$-linear $\Hom$-finite triangulated category $\cc$ is {\em $d$-Calabi--Yau} if it is endowed with bifunctorial isomorphisms
\[
\Hom(X, Y) \xlongrightarrow{_\sim} D\Hom(Y, \Si^d X) \: ,
\]
where $X$ and $Y$ lie in $\cc$. By definition, a right $d$-Calabi--Yau structure on a pretriangulated dg category $\ca$ yields a $d$-Calabi--Yau structure on the triangulated category $H^0(\ca)$.

The construction of the mixed complex $M(A)$ is functorial with respect to (not necessarily unital!) morphisms between dg algebras as defined in section~3.4 of \cite{KellerLiu23}. For a morphism $f\colon B\to A$ between dg algebras, its {\em relative mixed complex} $M(A,B)$ is defined to be the cone of the induced morphism from $M(B)$ to $M(A)$. The {\em relative} version of each type of complex and homology is obtained by applying the constructions of the absolute version to the relative mixed complex $M(A,B)$. The construction of the relative mixed complex and the associated homologies for a morphism between pseudo-compact dg algebras is analogous.

For a morphism $f\colon B\to A$ between smooth pseudo-compact dg algebras, a {\em left relative $d$-Calabi--Yau structure} on $f$ is a class $[(\tilde{\xi_A},\si \tilde{\xi_B})]$ in $HN_d(A,B)$ which is {\em non-degenerate}, \ie the morphism $[\hat{\xi_B}\si^{1-d}]\colon \Si^{d-1}B^\vee \to B$ in $\cd(B^e)$ and the morphism of triangles
\[
\begin{tikzcd}
  \Si^{d-1}A^\vee\arrow{rr}{(-1)^{d-1}\Si^{d-1}\mu^\vee}\arrow{d}{\Si^{-1}[\hat{\xi}'']} & & \Si^{d-1}(A\lten_B A)^\vee\arrow{r}\arrow{d}{\mathbf{L}f^{e*}[\hat{\xi_B}\si^{1-d}]}& \Si^{d}\cone(\mu)^\vee\arrow{r}\arrow{d}{[\hat{\xi}']} & \Si^d A^\vee\arrow{d}{[\hat{\xi}'']}\\
  \Si^{-1}\cone(\mu)\arrow{rr} & & A\lten_{B}A\arrow[swap]{r}{\mu} & A\arrow{r} & \cone(\mu)
\end{tikzcd}
\]
in $\cd(A^e)$ obtained from the image $[(\xi_A,\si \xi_B)]$ of $[(\tilde{\xi_A},\si \tilde{\xi_B})]$ under the canonical map $HN_d(A,B)\to HH_d(A,B)$ are isomorphisms. Here we denote the induction functor $\cc(B^e) \to \cc(A^e)$ by $f^{e*}$.
Let us write $X(A)$ for a cofibrant resolution of $A$ as a pseudo-compact dg $A$-bimodule and similarly for $B$. We denote the morphism of degree $-d$ corresponding to $\xi_A\in A\ten_{A^e}X(A)$ by $\hat{\xi_A}\colon X(A)^\vee \to A$ and the morphism of degree $1-d$ corresponding to $\xi_B\in B\ten_{B^e}X(B)$ by $\hat{\xi_B}\colon X(B)^\vee \to B$. We denote by $\hat{\xi}'$ and $\hat{\xi}''$ the morphisms
\[
\begin{bmatrix}
(-1)^{1-d}\mu\circ f^{e*}(\hat{\xi_B}\si^{1-d}) & -\hat{\xi_A}\si^{-d}
\end{bmatrix}
\colon \Si^{d}\cone(\mu)^\vee \longrightarrow A
\]
respectively
\[
\begin{bmatrix}
(-1)^{d-1}\hat{\xi_A}\si^{-d} \\
(-1)^{d-1} \Si f^{e*}(\hat{\xi_B}\si^{1-d})\circ \Si^d\mu^\vee
\end{bmatrix}
\colon \Si^{d}X(A)^\vee \longrightarrow \cone(\mu)
\]
in $\cc(A^e)$ and use the identification $(A\lten_B A)^\vee \simeq A\lten_B B^\vee\lten_B A$. In particular, the class $[\tilde{\xi_B}]$ in $HN_{d-1}(B)$ is a left $(d-1)$-Calabi--Yau structure on $B$. If the dg algebra $B$ vanishes, then we recover the absolute notion.

For a morphism $f\colon A\to B$ between homologically finite dg algebras (not supposed to be pseudo-compact), a {\em right relative $d$-Calabi--Yau structure} on $f$ is a class $[(\si \tilde{x_B},\tilde{x_A})]$ in $DHC_{1-d}(B,A)$ which is {\em non-degenerate}, \ie the morphism $[\si^{1-d}\hat{x_B}]\colon B\to \Si^{1-d}DB$ in $\cd(B^e)$ and the morphism of triangles
\[
\begin{tikzcd}
  \cocone(f)\arrow{r}\arrow{d}{[\hat{x}'']}&A\arrow{r}{f}\arrow{d}{[\hat{x}']}&B\arrow{rr}\arrow{d}{f^e_*[\si^{1-d}\hat{x_B}]}& &\Si\cocone(f)\arrow{d}{\Si[\hat{x}'']}\\
  \Si^{-d}DA\arrow{r}&\Si^{-d}D\cocone(f)\arrow{r}&\Si^{1-d}DB\arrow[swap]{rr}{(-1)^{1-d}\Si^{1-d}Df}& &\Si^{1-d}DA
\end{tikzcd}
\]
in $\cd(A^e)$ obtained from the image $[(\si x_B,x_A)]$ of $[(\si \tilde{x_B},\tilde{x_A})]$ under the canonical map $DHC_{1-d}(B,A)\to DHH_{1-d}(B,A)$ are isomorphisms. Here we denote the restriction functor $\cc(B^e) \to \cc(A^e)$ by $f^e_*$.
Let us write $X(A)$ for a cofibrant resolution of $A$ as a dg $A$-bimodule and similarly for $B$. We denote the morphism of degree $-d$ corresponding to $x_A\in D(A\ten_{A^e}X(A))$ by $\hat{x_A}\colon X(A) \to DA$ and the morphism of degree $1-d$ corresponding to $x_B\in D(B\ten_{B^e}X(B))$ by $\hat{x_B}\colon X(B) \to DB$. We denote by $\hat{x}'$ and $\hat{x}''$ the morphisms
\[
\begin{bmatrix}
(-1)^{d-1}f^e_*(\si^{1-d}\hat{x_B})\circ f \\
(-1)^{d-1}\si^{-d}\hat{x_A}
\end{bmatrix}
\colon A \longrightarrow \Si^{-d}D\cocone(f)
\]
respectively
\[
\begin{bmatrix}
(-1)^{1-d}\Si^{-d}Df\circ \Si^{-1}f^e_*(\si^{1-d}\hat{x_B}) & -\si^{-d}\hat{x_A}
\end{bmatrix}
\colon \cocone(f) \longrightarrow \Si^{-d}DA
\]
in $\cc(A^e)$. In particular, the class $[\tilde{x_B}]$ in $DHC_{1-d}(B)$ is a right $(d-1)$-Calabi--Yau structure on $B$. If the dg algebra $B$ vanishes, then we recover the absolute notion.

\subsection{Bijection between left and right Calabi--Yau structures in the pseudo-compact setting} \label{ss:left and right}

If $A$ is a complete $l$-augmented pseudo-compact dg algebra, by Corollary~15.2 of \cite{VandenBergh15}, we have a quasi-isomorphism $M(A^!)\to \D M(A)$. Since the functor $\D\colon (\Mod k)\op \to \PC k$ is an equivalence, it yields a quasi-isomorphism $M(A)\to \D M(A^!)$ of (pseudo-compact) mixed complexes. Therefore, we have the canonical quasi-isomorphism
\[
HN(A)=\RHom_\La (k,M(A))\longrightarrow \RHom_\La (k,\D M(A^!))\simeq \D(k\lten_\La M(A^!))\simeq \D HC(A^!)\: .
\]

\begin{proposition} \label{prop:left and right}
Let $A$ be a smooth complete $l$-augmented pseudo-compact dg algebra. Then a class in $HN_d(A)$ is a left $d$-Calabi--Yau structure on $A$ if and only if its image under the $k$-linear bijection
\[
\begin{tikzcd}
HN_d(A)\arrow{r}{\sim} & \D HC_{-d}(A^!)=DHC_{-d}(A^!) 
\end{tikzcd}
\]
is a right $d$-Calabi--Yau structure on $A^!$.
\end{proposition}

\begin{remark} The necessity of the condition holds for arbitrary dg algebras (and categories), \cf Lemma~4.1 of \cite{Keller08d} and Theorem~3.1 of \cite{BravDyckerhoff19}.
\end{remark}

\begin{proof}[Proof of Proposition~\ref{prop:left and right}]
Since the composed bijection fits into the following commutative square
\[
\begin{tikzcd}
HN_d(A)\arrow{r}\arrow{d} & DHC_{-d}(A^!)\arrow{d} \\
HH_d(A)\arrow{r} & DHH_{-d}(A^!)\mathrlap{\: ,}
\end{tikzcd}
\]
it suffices to prove that the bottom horizontal map preserves and detects non-degeneracy. Let $[\xi]$ be a class in $HH_d(A)$. Since $A$ is a complete $l$-augmented pseudo-compact dg algebra, the category $\cd(A^e)\op$ is compactly generated by $(\pvd A^e)\op$ and the category $\pvd A^e$ is generated by the object $\Hom_k(l,l)$. Therefore, the morphism $\phi \colon \Si^d A^\vee \to A$ in $\cd(A^e)$ obtained from the $[\xi]$ is an isomorphism if and only if its image $\phi^*$ under the functor $\RHom_{A^e}(?,\Hom_k(l,l))$ is an isomorphism. 
We compose and precompose $\phi^*$ with natural isomorphisms to obtain the following 
morphism $t\colon$
\begin{align*}
\RHom_A(l,l) & \simeq \RHom_{A^e}(A,\Hom_k(l,l)) \\
& \to \RHom_{A^e}(\Si^d A^\vee,\Hom_k(l,l)) \\
& \simeq \RHom_{A\op}(\Si^d A^\vee \lten_A \, l,l) \\
& \simeq D\RHom_{A\op}(l,\Si^d l)\: .
\end{align*}
One checks that $t$ agrees with the morphism $A^! \to \Si^{-d}A^!$ corresponding to the
image of $[\xi]$ in $DHH_{-d}(A^!)$. This implies the claim.
\end{proof}

\subsection{Deformed dg preprojective algebras and Ginzburg--Lazaroiu morphisms} \label{ss:deformed dg preprojective algebras}

We introduce a special class of augmented pseudo-compact dg algebras and a special class of augmented morphisms between pseudo-compact dg algebras. Let $k$ be a field of characteristic $0$ and $d\geq 2$ an integer. For a dg algebra $A$, we use the notation $\Tr(A)$ for the quotient complex $A/[A,A]$ of $A$ by the subcomplex generated by the graded commutators. Let $l$ be a finite-dimensional semisimple $k$-algebra. Following \cite{VandenBergh15}, let $\si$ be the Casimir element associated with a given trace on $l$. Suppose that we are given a triple $(V_c,\eta,w)$ satisfying the following assumptions.

\begin{assumptions}\mbox{}
\begin{itemize}
\item[a)] $V_c$ is a pseudo-compact graded $l$-bimodule of finite total dimension concentrated in degrees $[2-d,0]$.
\item[b)] $\eta$ is a non-degenerate and graded anti-symmetric element of $V_c\ten_{l^e}V_c$ which is of degree $2-d$.
\item[c)] $w$ is an element of $\Tr(T_l V_c)$ which is of degree $3-d$ and only contains cubic and higher terms such that we have $\{w,w\}_{\omega_\eta}=0$.
\end{itemize}
\end{assumptions}

The {\em $d$-dimensional deformed dg preprojective algebra} associated with the above triple $(V_c,\eta,w)$ is the $l$-augmented pseudo-compact dg algebra
\[
\Pi=\Pi_d(l,V_c,\eta,w)=(T_l(V_c \oplus zl),d)\: ,
\]
where $z$ is an $l$-central indeterminate of degree $1-d$. The topology is the product topology. The unit and the augmentation maps are the natural ones. The differential is determined by
\[
d(v)=\{w,v\}_{\omega_\eta}\mbox{ for all }v\in V_c \quad \mbox{and}\quad d(z)=\si' \eta\, \si''\: .
\]
By section~9.2 of \cite{VandenBergh15}, the algebra $\Pi$ is an honest dg algebra. By the implication from (2) to (1) in Theorem~9.4 of \cite{VandenBergh15}, the $d$-dimensional deformed dg preprojective algebra $\Pi$ carries a left $d$-Calabi--Yau structure. Many examples of deformed dg preprojective algebras arise as deformed Calabi--Yau completions as introduced in \cite{Keller11b}.

Let $\ol{l}$ and $l_F$ be finite-dimensional semisimple $k$-algebras. Denote their product by $l$ and the canonical algebra inclusion $l_F \to l$ by $\phi$. Let $\ol{\si}$ and $\si_F$ be the Casimir elements associated with given trace on $\ol{l}$ respectively $l_F$. Denote their sum by $\si$. Suppose that we are given a quintuple $(N,F,\eta,w,w_F)$ satisfying the following assumptions.
\begin{assumptions}\mbox{}
\begin{itemize}
\item[a)] $F$ is a pseudo-compact graded $l_F$-bimodule of finite total dimension concentrated in degrees $[\frac{3-d}{2},0]$ and $N$ is a pseudo-compact graded $l$-bimodule of finite total dimension concentrated in degrees $[2-d,0]$.
\item[b)] $\eta$ is a non-degenerate and graded anti-symmetric element of $N\ten_{l^e}N$ which is of degree $2-d$.
\item[c)] $w_F$ is an element of $\Tr(T_{l_F}(F\oplus R))$ which is of degree $4-d$ and only contains cubic and higher terms such that we have $\{w_F, w_F\}_{\omega_{\eta_F}}=0$. $w$ is an element of $\Tr(T_l(F\oplus N))$ which is of degree $3-d$ and only contains cubic and higher terms such that the element $w_F+\{w_F,w\}_{\omega_{\eta_F}}+\frac{1}{2}\{w,w\}_{\omega_\eta}$ lies in the kernel of the canonical surjection $\Tr(T_l(F\oplus R\oplus N)) \to \Tr(T_l(F\oplus N))$.
\end{itemize}
\end{assumptions}

Here we define $R=\Si^{d-3}DF$ and denote by $\eta_F$ the image of the identity $\id_{F}$ under the composed map
\[
\begin{tikzcd}
 \Hom_k(F,F) & F\ten_k DF \arrow[swap]{l}{\sim} \arrow{r} & (F\oplus R)\ten_{l_F^e}(F\oplus R)\: ,
 \end{tikzcd}
\]
where the first map is the canonical graded $k$-linear bijection and the second map sends $a\ten b$ to
\[
(-1)^{(d-3)|a|}a\ten s^{d-3}b-(-1)^{|a||b|}s^{d-3}b\ten a\: .
\]
Clearly, the element $\eta_F$ is non-degenerate and graded anti-symmetric of degree $3-d$.

The {\em $d$-dimensional Ginzburg--Lazaroiu morphism} associated with the above quintuple $(N,F,\eta,w,w_F)$ is the $\phi$-augmented morphism
\[
\gamma=\gamma_d(\phi,N,F,\eta,w,w_F) \colon (T_{l_F}(F\oplus R\oplus z_F\, l_F),d)\longrightarrow (T_l(F\oplus N\oplus z\ol{l}),d)
\]
between pseudo-compact dg algebras, where $z_F$ is an $l_F$-central indeterminate of degree $2-d$ and $z$ is an $l$-central indeterminate of degree $1-d$ whose annihilator is $\ol{l}$. The topology of both pseudo-compact graded algebras are the product topology. The unit and the augmentation maps of both augmented pseudo-compact graded algebras are the natural ones. The differential of $T_{l_F}(F\oplus R\oplus z_F\, l_F)$ is determined by
\[
d(v)=\{w_F,v\}_{\omega_{\eta_F}}\mbox{ for all }v\in F\oplus R\quad \mbox{and}\quad d(z_F)=\si_F'\, \eta_F\, \si_F''\: .
\]
The differential of $T_l(F\oplus N\oplus z\ol{l})$ is determined by
\[
d(v)=\{w_F,v\}_{\omega_{\eta_F}}\mbox{ for all }v\in F\ko d(v)=\{w,v\}_{\omega_\eta}\mbox{ for all }v\in N\quad\mbox{and}\quad d(z)=\ol{\si}'\eta \,\ol{\si}''\: .
\]
The morphism $\gamma$ is determined by
\[
\gamma(v)=v\mbox{ for all }v\in F\ko \gamma(v)=-\{w,v\}_{\omega_{\eta_F}}\mbox{ for all }v\in R\quad \mbox{and}\quad \gamma(z_F)=\si_F'\, \eta \, \si_F''\: .
\]
By Propositions~4.2.3, 4.2.4, 4.2.5, Lemmas~4.2.6 and 4.2.7 of \cite{KellerLiu23}, the source and target of $\gamma$ are honest algebras and $\gamma$ is an honest morphism between dg algebras. By the implication from i) to ii) in Theorem~4.3.1 of \cite{KellerLiu23}, the $d$-dimensional Ginzburg--Lazaroiu morphism $\gamma$ carries a left relative $d$-Calabi--Yau structure. Many examples of Ginzburg--Lazaroiu morphisms arise as deformed relative Calabi--Yau completions as introduced in \cite{Yeung16} and studied, for example, in \cite{BozecCalaqueScherotzke20, Wu23}. The target $\Pi=\Pi_d(\phi,N,F,\eta,w,w_F)$ of $\gamma$ is called the {\em $d$-dimensional deformed relative dg preprojective algebra} associated with the quintuple $(N,F,\eta,w,w_F)$. If the algebra $l_F$ vanishes, then we recover the absolute notion.

\subsection{Cluster categories and Higgs categories} \label{ss:cluster categories and Higgs categories}

Let $d$ be a positive integer and $\ce$ an extriangulated category. We refer the reader to section~3.1 of \cite{GorskyNakaokaPalu21} and section~3.4 of \cite{Chen24} for the definition of the higher extension bifunctors $\mathbb{E}^i(?,-)$. In the sequel, we use them freely. A subcategory $\ct \subseteq \ce$ is {\em $d$-cluster-tilting} if it is functorially finite and we have 
\begin{align*}
\ct &=\{X\in \ce \mid \mathbb{E}^i(T,X)=0\mbox{ for all objects }T\in \ct\mbox{ and }i=1,\ldots,d-1\} \\
     &=\{X\in \ce \mid \mathbb{E}^i(X,T)=0\mbox{ for all objects }T\in \ct\mbox{ and }i=1,\ldots,d-1\}\: .
\end{align*}
An object $T$ in $\ce$ is {\em $d$-cluster-tilting} if the subcategory $\add T\subseteq \ce$ is 
$d$-cluster-tilting. In particular, an object is $1$-cluster-tilting if and only if it is an additive generator.
A subcategory $\ct \subseteq \ce$ is {\em $d$-rigid} if we have $\mathbb{E}^i(T,T')=0$ for all objects $T,T'\in \ct$ and all integers $i=1$, $\ldots$, $d-1$.

Now we assume that the field $k$ is of characteristic $0$. Let $\Pi=\Pi_{d+1}(l, V_c,\eta,w)$ be a $(d+1)$-dimensional deformed dg preprojective algebra. The associated {\em cluster category} $\cc_\Pi$ is defined \cite{Amiot09,Plamondon11,Guo11} as the idempotent completion of the Verdier quotient $\per \Pi/\pvd \Pi$. If $d$ is greater than or equal to $2$ and $H^0(\Pi)$ is finite-dimensional, then the Verdier quotient $\per \Pi/\pvd \Pi$ is Karoubian, ($\Hom$-finite,) $d$-Calabi--Yau and contains $\Pi$ as a canonical \linebreak $d$-cluster-tilting object with endomorphism algebra isomorphic to $H^0(\Pi)$, \cf Theorem~2.2 of \cite{Guo11}.

Let $\gamma=\gamma_{d+1}(\phi,N,F,\eta,w,w_F)$ be a $(d+1)$-dimensional Ginzburg--Lazaroiu morphism with target $\Pi$. The associated {\em relative cluster category} $\cc_\gamma$ is defined \cite{Wu23,KellerWu23} as the idempotent 
completion of the Verdier quotient $\per \Pi/\pvd\!_F \Pi$, where $\pvd\!_F \Pi$ denotes the triangulated subcategory of $\pvd \Pi$ generated by the dg $\Pi$-modules which restrict to zero on the source of $\gamma$. If $d$ is greater than or equal to $2$ and $H^0(\Pi)$ is finite-dimensional, then the Verdier quotient $\per \Pi/\pvd\!_F \Pi$ is Karoubian and $\Hom$-finite, \cf Corollaries~4.17 and 4.19 of \cite{Wu23}. If the algebra $l_F$ vanishes, then we recover the absolute notion.
If $d$ equals $1$, the associated {\em Higgs category} $\ch_\gamma$ is defined as the full subcategory $\add \Pi$ of $\cc_\gamma$. If $d$ is greater than or equal to $2$, the associated {\em Higgs category} $\ch_\gamma$ is defined \cite{Wu23,KellerWu23} as the full subcategory of $\cc_\gamma$ whose objects are the $X$ such that the vector spaces $\cc_\gamma(\Pi,\Si^i X)$ are finite-dimensional for all $i=1$, $\ldots$, $d-1$ and there exist triangles
\[
\begin{tikzcd}
X_1 \arrow{r}{f_1} & M_0 \arrow{r} & X \arrow{r} & \Si X_1\: ,
\end{tikzcd}
\]
\[
\begin{tikzcd}
X_2 \arrow{r}{f_2} & M_1 \arrow{r} & X_1 \arrow{r} & \Si X_2\: ,
\end{tikzcd}
\]
\[
\ldots
\]
\[
\begin{tikzcd}
X_{d-2} \arrow{r}{f_{d-2}} & M_{d-3} \arrow{r} & X_{d-3} \arrow{r} & \Si X_{d-2}\: ,
\end{tikzcd}
\]
\[
\begin{tikzcd}
M_{d-1} \arrow{r}{f_{d-1}} & M_{d-2} \arrow{r} & X_{d-2} \arrow{r} & \Si M_{d-1}
\end{tikzcd}
\]
in $\cc_\gamma$ with all $M_i$ in $\add \Pi$ and all $\cc_\gamma(f_i,l_F\Pi)$ surjective, and there exist triangles
\[
\begin{tikzcd}
X \arrow{r} & N_0 \arrow{r}{g_1} & X'_1 \arrow{r} & \Si X\: ,
\end{tikzcd}
\]
\[
\begin{tikzcd}
X'_1 \arrow{r} & N_1 \arrow{r}{g_2} & X'_2 \arrow{r} & \Si X'_1 \: ,
\end{tikzcd}
\]
\[
\ldots
\]
\[
\begin{tikzcd}
X'_{d-3} \arrow{r} & N_{d-3} \arrow{r}{g_{d-2}} & X'_{d-2} \arrow{r} & \Si X'_{d-3}\: ,
\end{tikzcd}
\]
\[
\begin{tikzcd}
X'_{d-2} \arrow{r} & N_{d-2} \arrow{r}{g_{d-1}} & N_{d-1} \arrow{r} & \Si X'_{d-2}
\end{tikzcd}
\]
in $\cc_\gamma$ with all $N_i$ in $\add \Pi$ and all $\cc_\gamma(l_F\Pi,g_i)$ surjective. The following properties are shown for the case of $d=2$ in \cite{KellerWu23} but their proofs for the case $d>2$ are similar: the Higgs category $\ch_\gamma$ is Karoubian, \cf Lemma~3.23 of \cite{KellerWu23}, and an extension closed subcategory of $\cc_\gamma$ hence an extriangulated category, \cf Proposition~3.26 of \cite{KellerWu23}. If the subcategory $\add l_F \Pi \subseteq \add \Pi$ is functorially finite and $H^0(\Pi/(l_F))$ is finite-dimensional, then $\ch_\gamma$ is a Frobenius extriangulated category with the full subcategory $\add l_F\Pi$ of projective-injective objects and contains $\Pi$ as a canonical $d$-cluster-tilting object with endomorphism algebra isomorphic to $H^0(\Pi)$. Its stable category is equivalent to the cluster category associated with the dg quotient $\Pi/(l_F)$ as an extriangulated category. If moreover, the homology of $\Pi$ is concentrated in degree $0$, then $\ch_\gamma$ is an exact category and the bounded derived category $\cd^b(\ch_\gamma)$ is triangle equivalent to the relative cluster category $\cc_\gamma$, \cf Theorem~4.18 of \cite{KellerWu23}.

\section{Calabi--Yau structures on Drinfeld quotients}

\subsection{Homotopy short exact sequences} \label{ss:htpy short exact sequences}

We define a {\em homotopy short exact sequence of complexes} to be a diagram
\[
\begin{tikzcd}
B\arrow[swap]{r}{i}\arrow[bend left]{rr}{h} & A\arrow[swap]{r}{p} & C
\end{tikzcd}
\]
of complexes of abelian groups such that $i$ and $p$ are morphisms of complexes and $h$ is a homogeneous morphism of degree $-1$ satisfying $d(h)=p\circ i$ and that the graded abelian
group $\Si^{-1}C\oplus A\oplus \Si B$ endowed with 
the differential
\[
\begin{bmatrix}
-\si^{-1}d\si & -\si^{-1}p & -\si^{-1}h\si^{-1} \\
0 & d & i\si^{-1} \\
0 & 0 & -\si d\si^{-1}
\end{bmatrix}
\]
is acyclic. Notice that the latter condition is equivalent to the morphism
\[
\begin{bmatrix}
p & h\si^{-1}
\end{bmatrix}
\colon \cone(i)\to C
\]
being a quasi-isomorphism, or the morphism
\[
\begin{bmatrix}
-\si^{-1}h & i
\end{bmatrix}
\colon B\to \cocone(p)
\]
being a quasi-isomorphism. Such a diagram yields a triangle
\[
\begin{tikzcd}
B\arrow{r} & A\arrow{r} & C\arrow{r}{\delta} & \Si B
\end{tikzcd}
\]
in the derived category of abelian groups, where the connecting morphism $\delta$ is represented by the (right) fraction
\[
\begin{tikzcd}[ampersand replacement=\&]
C \& \& \cone(i)\arrow{ll}[swap]{\begin{bsmallmatrix} p & h\sigma^{-1} \end{bsmallmatrix}}\arrow{rr}{\begin{bsmallmatrix} 0 & -\id_{\Si B} \end{bsmallmatrix}} \& \& \Si B \: .
\end{tikzcd}
\]
So we can compute the connecting morphism
\[
\delta \colon H^q(C) \to H^{q+1}(B)
\]
thanks to the following homotopy snake lemma.
\begin{lemma} \label{lemma:homotopy snake lemma}
Suppose that $b$ lies in $B^{q+1}$ and $c$ lies in $C^q$ such that we have $d(b)=0$ and $d(c)=0$. Then we have $\delta(\ol{c})=-\ol{b}$ if and only if there is an element $a$ in $A$ such that we have $d(a)+i(b)=0$ and $\ol{p(a)+h(b)}=\ol{c}$.
\end{lemma}

\subsection{Homotopy short exact sequences from Drinfeld quotients}

Let $\ca$ be a dg category and $\cb\subseteq \ca$ a full dg subcategory. By definition \cite{Drinfeld04}, the {\em Drinfeld quotient} \linebreak $\ca/\cb=\ca/_{Dr}\,\cb$ is obtained from $\ca$ by formally adjoining a contracting homotopy \linebreak $h_N\colon N \to N$ for each object $N$ in $\cb$, where $h_N$ is of degree $-1$ satisfying $d(h_N)=\id_N$.
As shown in \cite{Drinfeld04}, the sequence of dg categories
\[
\begin{tikzcd}
0 \arrow{r} & \cb \arrow{r} & \ca \arrow{r} & \ca/\cb \arrow{r} & 0
\end{tikzcd}
\]
is {\em exact}, \ie the sequence of triangulated categories
\[
\begin{tikzcd}
0 \arrow{r} &\cd(\cb) \arrow{r} & \cd(\ca) \arrow{r} & \cd(\ca/\cb) \arrow{r} & 0
\end{tikzcd}
\]
is exact.

\begin{proposition} \label{prop:htpy short exact sequence of resolutions}
Let $\ca$ be a small dg category and $\cb\subseteq\ca$ a full dg subcategory. Define $\cc=\ca/\cb$ to
be the Drinfeld quotient of $\ca$ by $\cb$ so that we have an exact sequence of dg categories
\[
\begin{tikzcd}
0 \arrow{r} & \cb \arrow{r}{I} & \ca \arrow{r}{Q} & \cc \arrow{r} & 0\: .
\end{tikzcd}
\]
Then the diagram
\begin{equation} \label{eq:htpy short exact sequence of resolutions}
\begin{tikzcd}
I^{e*}(B(\cb)) \arrow[bend left]{rr}{\tilde{h}} \arrow[swap]{r}{\tilde{i}} & B(\ca) \arrow[swap]{r}{\tilde{p}} & Q^e_*(B(\cc))
\end{tikzcd}
\end{equation}
is a homotopy short exact sequence of dg $\ca$-bimodules, where $\tilde{i}$ and $\tilde{p}$ denote the natural morphisms induced by $I$ respectively $Q$ and $\tilde{h}$ sends
\[
f_n\ten \cdots \ten f_{-1}
\]
in $\ca(B_n,-)\ten_k \cdots \ten_k \ca(?,B_0)$ to
\[
\sum_{i=0}^n (-1)^{i+1+|f_{n-i}|+\cdots+|f_n|}f_n \ten\cdots\ten f_{n-i}\ten h_{B_{n-i}}f_{n-i-1}\ldots h_{B_0}f_{-1}
\]
in $\cc(B_n,Q-)\ten_k \cdots \ten_k \cc(Q\, ?,B_0)$ for all $B_0$, $\ldots$, $B_n$ in $\cb$.
\end{proposition}
\begin{proof}
For simplicity, we write $h$ for the contracting homotopy associated with any object. For any $f_n\ten \cdots \ten f_{-1}$ in $\ca(B_n,-)\ten_k \cdots \ten_k \ca(?,B_0)$, we have
\begin{align*}
 & d(\tilde{h}(f_n\ten\cdots\ten f_{-1})) \\
= & d(\sum_{i=0}^n (-1)^{i+1+|f_{n-i}|+\cdots+|f_n|}f_n\ten\cdots\ten f_{n-i}\ten hf_{n-i-1}\ldots hf_{-1}) \\
= & \sum_{0\leq j<i\leq n}(-1)^{i+j+1+|f_{n-i}|+\cdots+|f_{n-j-1}|}f_n\ten\cdots\ten f_{n-j}f_{n-j-1}\ten\cdots\ten f_{n-i}\ten hf_{n-i-1}\ldots hf_{-1} \\
 & -\sum_{i=1}^n f_n\ten\cdots\ten f_{n-i+1}\ten f_{n-i}hf_{n-i-1}\ldots hf_{-1} \\
 & +\sum_{0\leq j\leq i\leq n}(-1)^{i+j+1+|f_{n-i}|+\cdots+|f_{n-j}|}f_n\ten\cdots\ten d(f_{n-j})\ten\cdots\ten f_{n-i}\ten hf_{n-i-1}\ldots hf_{-1} \\
 & +\sum_{0\leq i\leq j\leq n}(-1)^{i+j+|f_{n-j}|+\cdots+|f_{n-i-1}|}f_n\ten\cdots\ten f_{n-i}\ten hf_{n-i-1}\ldots hf_{n-j}f_{n-j-1}\ldots hf_{-1} \\
 & +\sum_{0\leq i<j\leq n+1}(-1)^{i+j+|f_{n-j+1}|+\cdots+|f_{n-i-1}|}f_n\ten\cdots\ten f_{n-i}\ten hf_{n-i-1}\ldots hd(f_{n-j})\ldots hf_{-1}
\end{align*}
and
\begin{align*}
 & \tilde{h}(d(f_n\ten\cdots\ten f_{-1})) \\
= & \tilde{h}(\sum_{j=0}^n (-1)^{j+|f_{n-j}|+\cdots+|f_n|}f_n\ten\cdots\ten f_{n-j}f_{n-j-1}\ten\cdots\ten f_{-1} \\
 & +\sum_{j=0}^{n+1}(-1)^{j+|f_{n-j+1}|+\cdots+|f_n|}f_n\ten\cdots\ten d(f_{n-j})\ten\cdots\ten f_{-1}) \\
= & \sum_{0\leq i<j\leq n}(-1)^{i+j+1+|f_{n-j}|+\cdots+|f_{n-i-1}|}f_n\ten\cdots\ten f_{n-i}\ten hf_{n-i-1}\ldots hf_{n-j}f_{n-j-1}\ldots hf_{-1} \\
 & +\sum_{0\leq j<i\leq n}(-1)^{i+j+|f_{n-i}|+\cdots+|f_{n-j-1}|}f_n\ten\cdots\ten f_{n-j}f_{n-j-1}\ten\cdots\ten f_{n-i}\ten hf_{n-i-1}\ldots hf_{-1} \\
 & +\sum_{0\leq i<j\leq n+1}(-1)^{i+j+1+|f_{n-j+1}|+\cdots+|f_{n-i-1}|}f_n\ten\cdots\ten f_{n-i}\ten hf_{n-i-1}\ldots hd(f_{n-j})\ldots hf_{-1} \\
 & +\sum_{0\leq j\leq i\leq n}(-1)^{i+j+|f_{n-i}|+\cdots+|f_{n-j}|}f_n\ten\cdots\ten d(f_{n-j})\ten\cdots\ten f_{n-i}\ten hf_{n-i-1}\ldots hf_{-1}\: .
\end{align*}
We deduce that
\begin{align*}
d(\tilde{h})(f_n\ten \cdots \ten f_{-1})= & (d\circ \tilde{h}+\tilde{h}\circ d)(f_n\ten \cdots \ten f_{-1}) \\
= & -\sum_{i=1}^n f_n\ten\cdots\ten f_{n-i+1}\ten f_{n-i}hf_{n-i-1}\ldots hf_{-1} \\
 & +\sum_{i=0}^n f_n\ten\cdots\ten f_{n-i}\ten f_{n-i-1}hf_{n-i-2}\ldots hf_{-1} \\
= & f_n\ten\cdots\ten f_{-1} \: .
\end{align*}
This means that we have $d(\tilde{h})=\tilde{p}\circ \tilde{i}$.
To prove that the morphism
\[
\begin{bmatrix}
\tilde{p} & \tilde{h}\si^{-1}
\end{bmatrix}
\colon \cone(\tilde{i})\to Q^e_*(B(\cc))
\]
in $\cd(\ca^e)$ is an isomorphism, it suffices to prove that the morphism
\[
\phi_L=L\lten_{\ca}
\begin{bmatrix}
\tilde{p} & \tilde{h}\si^{-1}
\end{bmatrix}\colon L\lten_\ca \cone(\tilde{i})\longrightarrow L\lten_\ca Q^e_*(B(\cc))
\]
is an isomorphism for all $L$ in $\cd(\ca)$. Since the exact sequence of dg categories
\[
\begin{tikzcd}
0 \arrow{r} & \cb \arrow{r}{I} & \ca \arrow{r}{Q} & \cc \arrow{r} & 0
\end{tikzcd}
\]
yields a semi-orthogonal decomposition of the derived category $\cd(\ca)$, there is a triangle
\[
\begin{tikzcd}
(\mathbf{L}I^*\circ \mathbf{R}I_*)L \arrow{r} & L \arrow{r} & (\mathbf{R}Q_*\circ \mathbf{L}Q^*)L \arrow{r} & \Si(\mathbf{L}I^*\circ \mathbf{R}I_*)L
\end{tikzcd}
\]
for all $L$ in $\cd(\ca)$. So it suffices to prove that $\phi_L$ is an isomorphism for all $L$ which lie in the image of the functor $\mathbf{L}I^*$ or that of the functor $\mathbf{R}Q_*$. Moreover, since the derived categories $\cd(\cb)$ and $\cd(\cc)$ equal the localising subcategories generated by the dg $\cb$-modules $\cb(?,B)$, $B\in \cb$, respectively dg $\cc$-modules $\cc(?,C)$, $C\in \cc$, and the functors $\mathbf{L}I^*$, $\mathbf{R}Q_*$, $?\, \lten_\ca \cone(\tilde{i})$, $?\, \lten_\ca Q^e_*(B(\cc))$ are triangle functors which commute with arbitrary coproducts,
it suffices to prove that $\phi_L$ is an isomorphism for $L=\mathbf{L}I^*(\cb(?,B))=\ca(?,B)$ and for $L=\mathbf{R}Q_*(\cc(?,C))=\cc(Q\, ?,C)$, where $B\in\cb$ and $C\in\cc$. In the first case, the morphism $L\lten_\ca \, \tilde{i}$ is 
the composed isomorphism
\[
\begin{tikzcd}
\ca(?,B)\ten_\ca I^{e*}(B(\cb)) \arrow{r}{\sim} & \ca(?,B) & \ca(?,B)\ten_\ca B(\ca)\arrow[swap]{l}{\sim}
\end{tikzcd}
\]
in $\cd(\ca)$. Thus, its cone $L\lten_\ca \cone(\tilde{i})$ vanishes. On the other hand, the dg $\ca$-module
\[
\begin{tikzcd}
L\lten_\ca Q^e_*(B(\cc))=\ca(?,B)\ten_\ca Q^e_*(B(\cc)) \arrow{r}{\sim} & \cc(Q\, ?,B)
\end{tikzcd}
\]
in $\cd(\ca)$ vanishes because $B$ is a contractible object in $\cc$. 
Therefore, the morphism $\phi_L$ is an isomorphism in this case. In the second case, the dg $\ca$-module
\[
L\lten_\ca I^{e*}(B(\cb))=\cc(Q\, ?,C)\ten_\ca I^{e*}(B(\cb))\iso \cc((Q\circ I)\, ?,C)\ten_\cb \ca=(\mathbf{L}I^*\circ \mathbf{R}I_*\circ \mathbf{R}Q_*)\, \cc(?,C)
\]
in $\cd(\ca)$ vanishes. This means that the canonical morphism $L\lten_\ca B(\ca)\to L\lten_\ca \cone(\tilde{i})$ is an isomorphism. On the other hand, the composition $L\lten_\ca \, \tilde{p}$ of $\phi_L$ with the canonical morphism $L\lten_\ca B(\ca)\to L\lten_\ca \cone(\tilde{i})$ is the composed isomorphism
\[
\begin{tikzcd}
\cc(Q\, ?,C)\ten_\ca B(\ca) \arrow{r}{\sim} & \cc(Q\, ?,C) & \cc(Q\, ?,C)\ten_\ca Q^e_*(B(\cc))\arrow[swap]{l}{\sim}
\end{tikzcd}
\]
in $\cd(\ca)$. Therefore, the morphism $\phi_L$ is also an isomorphism in this case. This implies the assertion.
\end{proof}

\begin{proposition} \label{prop:htpy short exact sequence of Hochschild complexes}
Let $\ca$, $\cb$, $\cc$, $I$, $Q$ be as in Proposition~\ref{prop:htpy short exact sequence of resolutions}.
Let $\nu$ be a dg endofunctor of $\ca$ which takes $\cb$ to itself.
Let $M$ be the $\ca$-bimodule determined by $M(X,Y)=\ca(\nu X,Y)$.
Denote $I^e_* M$ by $M_\cb$ and $Q^{e*} M$ by $M^\cc$. Then the diagram
\begin{equation} \label{eq:htpy short exact sequence of Hochschild complexes}
\begin{tikzcd}
M_\cb \ten_{\cb^e}B(\cb) \arrow[bend left]{rr}{h}\arrow[swap]{r}{i} & M\ten_{\ca^e}B(\ca) \arrow[swap]{r}{p} & M^\cc \ten_{\cc^e}B(\cc)
\end{tikzcd}
\end{equation}
is a homotopy short exact sequence of complexes, where $i$ and $p$ denote the natural morphisms induced by $I$ respectively $Q$ and $h$ sends
\[
f_n\ten \cdots \ten f_0
\]
in $\cb(\nu B_n,B_0)\ten_k \cb(B_{n-1},B_n)\ten_k \cdots \ten_k \cb(B_0,B_1)$ to
\[
\sum_{i=0}^n (-1)^{(i+1+|f_{n-i}|+\cdots+|f_n|)(n-i+|f_0|+\cdots+|f_{n-i-1}|)}h_{B_{n-i}}f_{n-i-1}\ldots h_{B_1}f_0 h_{B_0}f_n\ten f_{n-1}\ten\cdots\ten f_{n-i}
\]
in $\cc(\nu B_n,B_0)\ten_k \cc(B_{n-1},B_n)\ten_k \cdots \ten_k \cc(B_0,B_1)$ for all $B_0$, $\ldots$, $B_n$ in $\cb$.
\end{proposition}
\begin{proof}
Since $\cc$ is a Drinfeld quotient of $\ca$, the dg $\ca$-module $\cc(Q\, ?,C)$ and dg $\ca\op$-module $\cc(C,Q\, ?)$ are cofibrant for all $C$ in $\cc$. Therefore, the dg 
$\ca^e$-module $Q^e_*(B(\cc))$ is cofibrant.
So the terms of the sequence~(\ref{eq:htpy short exact sequence of resolutions})
are cofibrant over $\ca^e$ and its sum total complex is cofibrant and acyclic hence contractible. Therefore,
the image of the sequence~(\ref{eq:htpy short exact sequence of resolutions}) under the dg functor $M\ten_{\ca^e}\, ?$ is still homotopy short exact. Now the assertion follows from the isomorphisms
\[
\begin{tikzcd}
M\ten_{\ca^e}I^{e*}(B(\cb)) \arrow[no head]{r}{\sim} & M\ten_{\ca^e}\ca^e\ten_{\cb^e}B(\cb) \arrow{r}{\sim} & M_\cb\ten_{\cb^e}B(\cb)
\end{tikzcd}
\]
and
\[
\begin{tikzcd}
M\ten_{\ca^e}Q^e_*(B(\cc)) & M\ten_{\ca^e}\cc^e \ten_{\cc^e}B(\cc) \arrow[swap]{l}{\sim}\arrow[no head]{r}{\sim} & M^\cc \ten_{\cc^e}B(\cc)
\end{tikzcd}
\]
of complexes.
\end{proof}

\subsection{Amiot's construction and connecting morphisms}

If the dg category $\ca$ and its full dg subcategory $\cb\subseteq \ca$ are pretriangulated, then so is $\ca/\cb$ and we have an induced exact sequence of triangulated categories
\[
\begin{tikzcd}
0 \arrow{r} & H^0(\cb)\arrow{r} & H^0(\ca) \arrow{r} & H^0(\ca/\cb) \arrow{r} & 0 \: .
\end{tikzcd}
\]
Any morphism from $X$ to $Y$ in the Verdier quotient $H^0(\ca)/H^0(\cb)$ is represented by a fraction $f\circ s^{-1}$, where $s\colon X'\to X$ and $f\colon X'\to Y$ are morphisms in $H^0(\ca)$ such that $s$ fits into a triangle
\[
\begin{tikzcd}
\Si^{-1}N\arrow{r}{\pi} & X'\arrow{r}{s} & X\arrow{r}{\iota} & N
\end{tikzcd}
\]
with $N$ lying in $H^0(\cb)$. Thanks to the following lemma, the image of $f\circ s^{-1}$ under the canonical map
\[
\begin{tikzcd}
H^0(\ca)/H^0(\cb)(X,Y)\arrow{r}{\sim} & H^0(\ca/\cb)(X,Y)\arrow[no head]{r}{\sim} & H^{-1}(\ca/\cb)(X,\Si Y)
\end{tikzcd}
\]
is $\Si f\circ \Si \pi \circ h_N \circ \iota$. 

\begin{lemma}[Kontsevich, \cf Drinfeld \cite{Drinfeld04}] \label{lem:contracting homotopy}
Let $\ca$ be a pretriangulated dg category. 
Suppose that $f\colon X\to Y$ is a closed morphism of degree $0$ in $\ca$
and $\cone(f)$ is its cone with underlying graded object $Y\oplus \Si X$.
Let
\[
h=
\begin{bmatrix}
h_{11} & h_{12}\si^{-1} \\
\si h_{21} & -\si h_{22}\si^{-1}
\end{bmatrix}
\colon \cone(f)\to \cone(f)
\]
be an endomorphism of $\cone(f)$ in $\ca$ which is of degree $-1$.
Then $h$ is a contracting homotopy for $\cone(f)$ if and only if we have
\[
d\circ h_{21}=h_{21}\circ d\ko \id_X-h_{21}\circ f=d(h_{22})\ko \id_Y-f\circ h_{21}=d(h_{11})\ko f\circ h_{22}-h_{11}\circ f=d(h_{12})\: .
\]
\end{lemma}

\begin{remark} Thus, with the notations of the lemma, the morphism $f$ becomes invertible in $H^0(\ca)$ if and only if there is a morphism $h$ as in the lemma and in this case, the morphism $h_{21}$ is an inverse up to homotopy of $f$.
\end{remark}

Let us briefly recall Amiot's construction from section~1.1 of \cite{Amiot09} using the notations above summed up in the following diagram (notice that our $\nu$ corresponds to Amiot's $\nu^{-1}$)
\[
\begin{tikzcd}
\nu N \arrow{rr}{\nu\Si\pi} 	          & & \nu\Si X'\arrow{d}{\nu\Si f} & & & & \\
X'\arrow{rr}{s}\arrow[swap]{d}{f} & & \nu X\arrow{rr}{\iota}\arrow[bend left,dashed]{ll}{\si^{-1}\circ \Si\pi \circ h_N \circ \iota} 
& & N\arrow{rr}{\Si\pi}\arrow[out=225,in=315,loop,swap]{}{h_N}  & & \Si X'\\
\Si^{-1}X\mathrlap{\: .} & & & & & &
\end{tikzcd}
\]
In Amiot's setup, we are given linear forms $\tr_L \colon H^0(\cb)(\nu L,L)\to k$, $L \in H^0(\cb)$.
We then define a linear form $\tr_X \colon H^0(\ca)/H^0(\cb)(\nu X, \Si^{-1} X) \to k$ for each object $X$ in the quotient category $H^0(\ca)/H^0(\cb)$ by sending a fraction $f \circ s^{-1}$ as above to the scalar
\[
\tr_N(\iota \circ \nu\Si f \circ \nu\Si \pi)\: .
\]
In the following theorem, we lift this construction to the dg level
and link it to the connecting morphism in Hochschild homology.

\begin{theorem} \label{thm:connecting morphism}
We keep the notations in Proposition~\ref{prop:htpy short exact sequence of Hochschild complexes}. Assume that the dg categories $\ca$ and $\cb$ are pretriangulated. Then the square
\[
\begin{tikzcd}
H^{-1}(\coprod_{X\in \cc} M^\cc(X,X))\arrow{r}\arrow{d} & H^0(\coprod_{N\in \cb} M_\cb(N,N))\arrow{d} \\
HH_1(\cc,M^\cc)\arrow[swap]{r}{\delta} & HH_0(\cb,M_\cb)
\end{tikzcd}
\]
is commutative, where the maps are described as follows: the vertical ones are the canonical maps and the top horizontal one sends
$\Si f\circ \Si \pi \circ h_{N}\circ \iota$ to $\iota \circ \nu\Si f\circ \nu\Si \pi$,
where $s\colon X'\to \nu X$ and $f\colon X'\to \Si^{-1}X$ are morphisms in $H^0(\ca)$ such that 
$s$ fits into a triangle
\[
\begin{tikzcd}
\Si^{-1}N \arrow{r}{\pi} & X'\arrow{r}{s} & \nu X\arrow{r}{\iota} & N
\end{tikzcd}
\]
with $N$ lying in $H^0(\cb)$.
\end{theorem}

\begin{remark} \label{rk:compatibility} Thus, we see that if a class $x$ in $DHH_0(\cb, M_\cb)$ gives rise to linear forms
$\tr_N \colon H^0(\cb)(\nu N, N) \to k$, $N\in H^0(\cb)$, then its image $(D\delta)(x)$ gives rise to the linear forms
$\tr_X \colon H^0(\ca)/H^0(\cb)(\nu X, \Si^{-1}X) \to k$, $X \in H^0(\ca)/H^0(\cb)$, obtained from the \linebreak $\tr_N$, $N\in H^0(\cb)$, by Amiot's construction described above. This result was previously announced in \cite{KellerLiu23a}.
\end{remark}

\begin{proof}[Proof of theorem \ref{thm:connecting morphism}]
We first apply Lemma~\ref{lemma:homotopy snake lemma} to the homotopy short exact sequen\-ce~(\ref{eq:htpy short exact sequence of Hochschild complexes}) in Proposition~\ref{prop:htpy short exact sequence of Hochschild complexes} for the elements $a=\iota \ten(\Si f\circ \Si \pi)$ of $HH(\ca,M)$, \linebreak $b=-\iota \circ \nu\Si f\circ \nu\Si \pi$ of $HH(\cb,M_\cb)$ and $c=\Si f\circ \Si \pi \circ h_N \circ \iota$ of $HH(\cc,M^\cc)$. Then the assertion follows from that the element $h(b)+p(a)-c=d((h_N \circ \iota)\ten(\Si f\circ \Si \pi))$ is a boundary.
\end{proof}

\section{The structure theorem in the absolute case}

\subsection{Pseudo-compact dg quotients}

Let $\ce$ be an exact pseudo-compact category with enough projective objects. Let $\cp \subseteq \ce$ be the full subcategory of projective objects and $\cm \subseteq \ce$ an additive subcategory containing $\cp$.

\begin{proposition} \label{prop:pseudo-compact dg categories}
The following dg quotients carry natural structures of pseudo-compact dg categories:
\begin{itemize}
\item[a)] the dg derived category $\cd^b_{dg}(\ce)$;
\item[b)] the dg quotient $\cc^b_{dg}(\cm)/\cc^b_{dg}(\cp)$;
\item[c)] if the category $\ce$ is Frobenius, the dg singularity category $\cd^b_{dg}(\ce)/\cd^b_{dg}(\cp)$.
\end{itemize}
\end{proposition}
\begin{proof}
a) By definition, the dg derived category $\ca=\cd^b_{dg}(\ce)$ is the dg quotient of the dg category of bounded complexes $\cc^b_{dg}(\ce)$ by its full dg subcategory of $\ce$-acyclic complexes. Denote by $\cb$ the dg category $\cc^{-,b}_{dg}(\cp)$ of right bounded
complexes with components in $\cp$ which are $\ce$-acyclic in all
degrees $p\ll 0$. We define a dg $\ca$-$\cb$-bimodule $X$ by putting
\[
X(P,M)=\cc_{dg}(\ce)(P,M)\: .
\]
It is not hard to check that $X$ is right quasi-representable and that the derived functor
$?\lten_\ca X$ induces an equivalence $\per \ca \iso \per \cb$. Thus, the dg derived category $\cd^b_{dg}(\ce)$ is naturally quasi-equivalent to the dg category $\cc^{-,b}_{dg}(\cp)$. This dg category is a full dg subcategory of the pseudo-compact dg category $\cc_{dg}(\cp)$ of complexes with components in $\cp$ and therefore inherits a pseudo-compact dg category structure.

b) Denote by $\ch^-(\cm)$ the homotopy category of right bounded complexes
with components in $\cm$ and by $\ch^-(\cp)$ its full subcategory whose objects are
the complexes with components in $\cp$. For each object $X$ of $\ch^-(\cm)$, we have
a triangle
\[
\begin{tikzcd}
\bp X \arrow{r} & X \arrow{r} & \ba X \arrow{r} & \Si \bp X \: ,
\end{tikzcd}
\]
where $\bp X \to X$ is a projective resolution of the right bounded complex $X$ and
$\ba X$ belongs to the subcategory $\ch^-_{ac}(\cm) \subseteq \ch^-(\cm)$ of $\ce$-acyclic complexes with components in $\cm$. This shows that the triangulated category $\ch^-(\cm)$
admits the semi-orthogonal decomposition $(\ch^-(\cp), \ch^-_{ac}(\cm))$. In particular,
the quotient functor induces an equivalence
\[
\ch^-(\cm)/\ch^-(\cp) \xlongrightarrow{\sim} \ch^-_{ac}(\cm)\: .
\]
Clearly, the dg subcategory $\cc^-_{ac,dg}(\cm)$ inherits the structure of pseudo-compact dg category from $\cc^-_{dg}(\cm) \subseteq \cc^-_{dg}(\ce)$. Thus, the dg quotient $\cc^-_{dg}(\cm)/\cc^-_{dg}(\cp)$ becomes a pseudo-compact dg category. One checks easily that the canonical dg functor
\[
\cc^b_{dg}(\cm)/\cc^b_{dg}(\cp) \to \cc^-_{dg}(\cm)/\cc^-_{dg}(\cp)
\]
is quasi-fully faithful, which implies the claim.

c) By definition, the dg singularity category is the dg quotient of the dg derived category $\ca=\cd^b_{dg}(\ce)$ by its full dg subcategory of bounded complexes with components in $\cp$. Denote by $\cb'=\cc_{ac,dg}(\cp)$ the dg category of $\ce$-acyclic complexes with components in $\cp$.
We define a dg $\ca$-$\cb'$-bimodule $X$ by putting
\[
X(P,M)=\cc_{dg}(\ce)(P,\Si M)\: .
\]
It is not hard to check that $X$ is right quasi-representable and that the derived functor $?\lten_\ca X$ induces an equivalence $\per(\cd^b_{dg}(\ce)/\cd^b_{dg}(\cp))\iso \per \cb'$, \cf the proof of Lemma~5.12 of \cite{HuaKeller18}. Thus, the dg singularity category of $\ce$ is naturally quasi-equivalent to the dg category $\cc_{ac,dg}(\cp)$. We conclude as in part~a).
\end{proof}

\subsection{The structure theorem}

The following theorem was previously announced in \cite{KellerLiu23a}.

\begin{theorem} \label{thm:main} Let $k$ be a field of characteristic $0$ and $d$ a positive integer. Let $\cc$ be a $k$-linear algebraic triangulated category which is Karoubian and $\Hom$-finite.
Assume that there is a small $k$-linear Frobenius exact pseudo-compact category $\ce$ with the full subcategory $\cp$ of projective-injective objects and the stable category $\cc$. Then the following are equivalent.
\begin{itemize}
\item[i)] There exists a $(d+1)$-dimensional deformed dg preprojective algebra
\[
\Pi=\Pi_{d+1}(l, V_c,\eta,w)
\]
and a triangle equivalence $\cc_\Pi \iso \cc$.
\item[ii)] The triangulated category $\cc$ contains a $d$-cluster-tilting object $T$ and the dg category $\cc_{dg}=\cd^b_{dg}(\ce)/\cd^b_{dg}(\cp)$ carries a right $d$-Calabi--Yau structure.
\end{itemize}
\end{theorem}

\begin{remark}
Thus, if statement ii) of the theorem holds, then the endomorphism algebra of $T$ is isomorphic to $H^0(\Pi)$.
\end{remark}

\begin{example} \label{ex:example}
This example serves to illustrate the construction of $\Pi_d(l, V_c, \eta, w)$ in the case where $k$ is not algebraically closed. Starting from the $\R$-species
\tikzset{
    lablp/.style={anchor=south, rotate=39, inner sep=.5mm},
    lablm/.style={anchor=south, rotate=-39, inner sep=.5mm},
}
\[
\begin{tikzcd}
	& \R\arrow[dr,"\R c\oplus\R ic" lablm] & \\
	\R\arrow[ur,"\R a" lablp] & & \C\arrow[ll,"\R b\oplus\R bi"]
\end{tikzcd}
\]
with the potential $W=abc$ we obtain the graded double $\R$-species
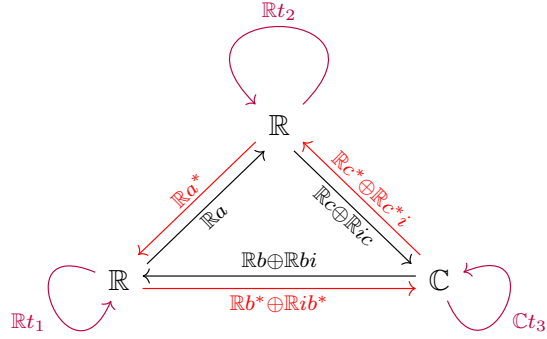
\tikzset{
    lablsouthp/.style={anchor=south, rotate=44, inner sep=.5mm},
    lablnorthp/.style={anchor=north, rotate=44, inner sep=.5mm},
    lablsouthm/.style={anchor=south, rotate=-44, inner sep=.5mm},
    lablnorthm/.style={anchor=north, rotate=-44, inner sep=.5mm},
}
\[
\begin{tikzcd}[row sep=1.5cm, column sep=1.5cm]
	& \R\arrow[dr,shift right=0.5ex,"\R c\oplus \R ic" lablnorthm,swap]
	\arrow[dl,shift right=0.5ex,"\R a^*" lablsouthp,swap,red]
	\arrow[out=45,in=135,loop,purple,"\R t_{2}",swap] & \\
	\R\arrow[ur,shift right=0.5ex,"\R a" lablnorthp,swap]
	\arrow[rr,shift right=0.5ex,"\R b^*\oplus \R ib^*",swap,red]
	\arrow[out=157,in=247,loop,purple,"\R t_{1}",swap] & &
	\C\arrow[ll,shift right=0.5ex,"\R b\oplus\R bi",swap]
	\arrow[ul,shift right=0.5ex,"\R c^*\oplus\R c^*i" lablsouthm,swap,red]
	\arrow[out=293,in=23,loop,purple,"\C t_{3}",swap]
\end{tikzcd}
\]
with the elements $a^*$, $b^*$, $c^*$ of degree $-1$ and $t_1$, $t_2$, $t_3$ of degree $-2$. It gives rise to the $l$-augmented pseudo-compact dg $\R$-algebra 
$\Pi=T_l(V_c\oplus\R t_1\oplus\R t_2\oplus\C t_3)$, where $l$ is the semisimple $\R$-algebra $\R\times\R\times\C$ and $V_c$ is the $l$-bimodule with the $\R$-basis
\[
a\ko b\ko bi\ko c\ko ic\ko a^*\ko b^*\ko ib^*\ko c^*\ko c^*i\: .
\]
In the notations in section~\ref{ss:deformed dg preprojective algebras}, we have $z=t_1+t_2+t_3$ and $\sigma$ is the Casimir element associated with the trace form $\Tr(\lambda_1,\lambda_2,\lambda_3)=\lambda_1+\lambda_2+\mathrm{Re}(\lambda_3)$.

For a suitable choice of non-degenerate and graded anti-symmetric element $\eta$ in $V_c\ten_{l^e}V_c$, we obtain the differential determined by
\[
d(t_1)=bb^*+bib^*-a^*a \ko
d(t_2)=aa^*-c^*c-c^*ic \ko
d(t_3)=(1+i)(cc^*-b^*b)+(1-i)(cc^*-b^*b)i
\]
and
\[
d(a^*) = -bc \ko
d(b^*) = -ca \ko
d(c^*) = -ab \: .
\]
Then the endomorphism algebra of the $2$-cluster-tilting object $P_1\oplus P_3\oplus\tau^{-2}P_1$ in the cluster category associated with the $\R$-species
\[
\begin{tikzcd}
	\R\arrow[r,"\R"]&\R\arrow[r,"\C"]&\C
\end{tikzcd}
\]
is isomorphic to $H^0(\Pi)$.
\end{example}

\subsection{Proof of the implication from i) to ii) in Theorem~\ref{thm:main}} \label{ss:proof from i) to ii) in main}

By Lemma~2.9 of \cite{Plamondon11}, the endomorphism algebra of $\Pi$ in $\cc_\Pi$ is isomorphic to $H^0(\Pi)$. Since the category $\cc$ is $\Hom$-finite, this implies that $H^0(\Pi)$ is finite-dimensional. If $d$ equals $1$, after possibly replacing $\Pi$ with a derived Morita equivalent pseudo-compact dg algebra, we may and will assume that it is the $2$-dimensional dg preprojective algebra associated with a species $S$ of Dynkin type \cite{DlabRingel74a}. By Theorem~1 and 2 of \cite{Keller05} and the proof of Theorem~4.10 of \cite{Amiot09}, there is a triangle equivalence $\cd^b(A)/\tau \iso \cc_\Pi$ taking $A$ to $\Pi$, where $A$ is the finite-dimensional hereditary $k$-algebra associated with $S$. In particular, the cluster category $\cc_\Pi$ contains a canonical $1$-cluster-tilting object. If $d$ is greater than or equal to $2$, by part~2) of Theorem~2.2 of \cite{Guo11}, the cluster category $\cc_\Pi$ contains a canonical $d$-cluster-tilting object. Since there is a triangle equivalence $\cc_\Pi \iso \cc$, the triangulated category $\cc$ also contains a $d$-cluster-tilting object. By the implication from (2) to (1) in Theorem~9.4 (also holds if the Calabi--Yau dimension is $2$) of \cite{VandenBergh15}, the pseudo-compact dg algebra $\Pi$ carries a left $(d+1)$-Calabi--Yau structure. By Proposition~\ref{prop:left and right}, it yields a right $(d+1)$-Calabi--Yau structure on $\Pi^!$. Since the dg category $\pvd\!_{dg}\Pi$ is quasi-equivalent to $\per\!_{dg}(\Pi^!)$ and the dg category $\per\!_{dg}(\Pi^!)$ is derived Morita equivalent to the dg algebra $\Pi^!$, we obtain a right $(d+1)$-Calabi--Yau structure $[\tilde{x}]$ on $\pvd\!_{dg}\Pi$. The exact sequence of dg categories
\[
\begin{tikzcd}
0\arrow{r} & \pvd\!_{dg}\Pi\arrow{r} & \per\!_{dg}\Pi\arrow{r} & \cc_{\Pi,dg}\arrow{r} & 0
\end{tikzcd}
\]
yields a connecting morphism $\delta \colon HC_{-d}(\cc_{\Pi,dg})\to HC_{-d-1}(\pvd\!_{dg}\Pi)$ in cyclic homology. We claim that the image $[x]$ of $[\tilde{x}]$ under the dual map
\[
D\delta \colon DHC_{-d-1}(\pvd\!_{dg}\Pi)\to DHC_{-d}(\cc_{\Pi,dg})
\]
is a right $d$-Calabi--Yau structure on $\cc_{\Pi,dg}$. This means that the bifunctorial morphism
\[
\begin{tikzcd}
\cc_{\Pi,dg}(X,Y)\arrow{r} & D\cc_{\Pi,dg}(Y,\Si^d X)
\end{tikzcd}
\]
in $\cd(k)$ obtained from the image of $[x]$ under the canonical map 
\[
\begin{tikzcd}
DHC_{-d}(\cc_{\Pi,dg})\arrow{r} & DHH_{-d}(\cc_{\Pi,dg})
\end{tikzcd}
\]
is an isomorphism. By the bifunctoriality, the induced morphism in homology is then determined by the linear forms 
\[
\psi_Z\colon \cc_\Pi(Z,\Si^d Z)\longrightarrow k \ko Z \in \cc_\Pi
\]
associated with $[x]$.

The right $(d+1)$-Calabi--Yau structure $[\tilde{x}]$ on $\pvd\!_{dg}\Pi$ defines a $(d+1)$-Calabi--Yau structure on the triangulated category $\pvd \Pi$ determined by the linear forms
\[
\tilde{\psi}_X \colon (\pvd \Pi)(X, \Si^{d+1}X) \longrightarrow k \ko X \in \pvd \Pi \: .
\]
By Remark~\ref{rk:compatibility} (based on Theorem~\ref{thm:connecting morphism}), the linear forms $\psi_Z$, $Z\in \cc_\Pi$, are obtained from the $\tilde{\psi}_X$, $X \in \pvd \Pi$, by Amiot's construction. Extending Theorem~2.1 of \cite{Amiot09}, it is showed in the proof of Theorem~2.2 of \cite{Guo11} that the linear forms $\psi_Z$, $Z\in \cc_\Pi$, determine a $d$-Calabi--Yau structure on the triangulated category $\cc_\Pi$. Therefore, the dg category $\cc_{\Pi,dg}$ carries a right $d$-Calabi--Yau structure and hence so does $\cc_{dg}$. This concludes the proof of the implication from i) to ii) in Theorem~\ref{thm:main}.

\subsection{Proof of the implication from ii) to i) in Theorem~\ref{thm:main}} \label{ss:proof from ii) to i) in main}

Denote by $\cm\subseteq \ce$ the closure under finite direct sums and direct summands of $T$ and the projective-injective objects in $\ce$. Then $\cm$ is a $d$-cluster-tilting subcategory of $\ce$. It is shown in Corollary~3.2 of \cite{KalckYang20} that $T$ is a silting object in $\ch^b(\cm)/\ch^b(\cp)$. We denote by $\ol{\Ga}$ its endomorphism algebra in the dg quotient $\cc^b_{dg}(\cm)/\cc^b_{dg}(\cp)$. By part~b) of Proposition~\ref{prop:pseudo-compact dg categories}, this dg quotient carries a pseudo-compact dg category structure and the dg algebra $\ol{\Ga}$ becomes a pseudo-compact dg algebra.

By Proposition~3, Lemma~7 of \cite{Palu09a} and part~(c) of Theorem~4.1 of \cite{KalckYang20}, 
the dg category $\cc_{dg}$ fits into the commutative diagram
\begin{equation} \label{eq:Palu's diagram 1}
\begin{tikzcd}
 & & 0 \arrow{d} & 0 \arrow{d} & \\
 & & \cc^b_{dg}(\cp) \arrow[equal]{r}\arrow{d} & \cc^b_{dg}(\cp) \arrow{d} &  \\
0 \arrow{r} & \pvd\!_{dg}\ol{\Ga} \arrow{r}\arrow[equal]{d} & \cc^b_{dg}(\cm) \arrow{r}\arrow{d} & \cd^b_{dg}(\ce)\arrow{r}\arrow{d} & 0 \\
0 \arrow{r} & \pvd\!_{dg}\ol{\Ga} \arrow{r} & \per\!_{dg}\ol{\Ga} \arrow{r}\arrow{d} & \cc_{dg}\arrow{r}\arrow{d} & 0 \\
 & & 0 & 0 &
\end{tikzcd}
\end{equation}
of dg categories with exact rows and columns and such that $T\in \cc^b_{dg}(\cm)$ is sent to \linebreak $\ol{\Ga}\in \per\!_{dg}\ol{\Ga}$. Notice that the exactness of the middle row shows that we can identify
$\pvd \ol{\Ga}$ with the full subcategory $\ch^b_{ac}(\cm)$ of $\per \cm$ formed by the bounded $\ce$-acyclic complexes with components in $\cm$.

Denote the radical quotient of $H^0(\ol{\Ga})$ by $l$. Since $k$ is of characteristic $0$, it is separable. Since $H^0(\ol{\Ga})=\End_\cc(T)$ is finite-dimensional, by Remark~\ref{rk:Wedderburn}, after possibly replacing $\ol{\Ga}$ with a quasi-isomorphic pseudo-compact algebra, we may and will assume that it is complete $l$-augmented pseudo-compact and concentrated in non-positive degrees. Then by the implication from ii) to i) in Corollary~\ref{cor:smooth}, it is smooth. By part~(c) of Theorem~1.5 of \cite{Keller98}, the commutative diagram (\ref{eq:Palu's diagram 1}) yields a morphism of long exact sequences in cyclic homology
\[
\begin{tikzcd}
\cdots \arrow{r} & HC_{-d}(\cd^b_{dg}(\ce))\arrow{r}\arrow{d} & HC_{-d-1}(\pvd\!_{dg}\ol{\Ga})\arrow{r}\arrow[equal]{d} & \cdots \\
\cdots \arrow{r} & HC_{-d}(\cc_{dg})\arrow[swap]{r}{\delta} & HC_{-d-1}(\pvd\!_{dg}\ol{\Ga})\arrow{r} & \cdots
\end{tikzcd}
\]
with the connecting morphism $\delta$.

\begin{proposition} \label{prop:detect non-degeneracy}
The dual map $D\delta \colon DHC_{-d-1}(\pvd\!_{dg}\ol{\Ga})\to DHC_{-d}(\cc_{dg})$ detects non-degeneracy.
\end{proposition}
\begin{proof}
Let $[\tilde{x}]$ be a class in $DHC_{-d-1}(\pvd\!_{dg}\ol{\Ga})$ such that its image $[x]$ under $D\delta$ is non-degenerate. We want to prove that $[\tilde{x}]$ is non-degenerate, 
which means that the bifunctorial morphism
\begin{equation} \label{eq: d+1 non-degeneracy 1}
\begin{tikzcd}
(\pvd\!_{dg}\ol{\Ga})(X,Y)\arrow{r} & D(\pvd\!_{dg}\ol{\Ga})(Y,\Si^{d+1}X)
\end{tikzcd}
\end{equation}
in $\cd(k)$ obtained from the image of $[\tilde{x}]$ under the canonical map 
\[
\begin{tikzcd}
DHC_{-d-1}(\pvd\!_{dg}\ol{\Ga})\arrow{r} & DHH_{-d-1}(\pvd\!_{dg}\ol{\Ga})
\end{tikzcd}
\]
is an isomorphism. It suffices to check that the morphism (\ref{eq: d+1 non-degeneracy 1}) is an isomorphism for simple (pseudo-compact) dg $\ol{\Ga}$-modules $X$ and $Y$ (which are concentrated in degree $0$) since these dg modules generate the triangulated category $\pvd \ol{\Ga}$. By the bifunctoriality, the induced morphism in homology is then determined by the linear forms 
\[
\tilde{\psi}_Z\colon \Ext^{d+1}_{\ol{\Ga}}(Z,Z)\longrightarrow k
\]
associated with $[\tilde{x}]$, where $Z$ runs through the simple dg $\ol{\Ga}$-modules.

Since, by assumption, the class $[x]$ is non-degenerate in $DHC_{-d}(\cc_{dg})$, it defines a \linebreak $d$-Calabi--Yau structure on the triangulated category $\cc$ determined by the linear forms
\[
\psi_X \colon \cc(X, \Si^d X) \longrightarrow k \ko X \in \cc \: ,
\]
obtained from the $\tilde{\psi}_Z$, $Z\in \pvd \ol{\Ga}$, by Amiot's construction, \cf Remark~\ref{rk:compatibility}.
Using this structure, from part~c) of Theorem~5.4 of \cite{KellerReiten07}, for a simple dg $\ol{\Ga}$-module $Z$, we obtain a functorial isomorphism
\begin{equation} \label{eq:KR}
\begin{tikzcd}
(\per\cm)(Z,?)\arrow{r}{\sim} & D(\per\cm)(?,\Si^{d+1}Z)\: .
\end{tikzcd}
\end{equation}
It is completely determined by the linear form 
\[
\phi_Z\colon (\per\cm)(Z,\Si^{d+1}Z) \longrightarrow k
\]
obtained as the image of the identity of $Z$. So it suffices to prove that $\tilde{\psi}_Z$ coincides with $\phi_Z$, \ie that the construction in the proof of part~c) of Theorem~5.4 of \cite{KellerReiten07} is essentially inverse to Amiot's.
For this, let us recall the construction of the isomorphism (\ref{eq:KR}): for any simple dg $\ol{\Ga}$-module $Z$, by the proof of part~b) of Theorem~5.4 of \cite{KellerReiten07}, there is a conflation
\begin{equation} \label{eq:conflation}
\begin{tikzcd}
0 \arrow{r} & K \arrow{r} & M_1 \arrow{r} & M_0 \arrow{r} & 0
\end{tikzcd}
\end{equation}
in $\ce$ with $M_0$ and $M_1$ in $\cm$ and an isomorphism between $Z$ and the cokernel of the morphism $\cm(?,M_1) \to \cm(?,M_0)$. By part~a) of Theorem~5.4 of \cite{KellerReiten07}, there is an $\ce$-acyclic complex
\begin{equation} \label{eq:complex}
\begin{tikzcd}
0 \arrow{r} & M_{d+1} \arrow{r} & M_d \arrow{r} & \cdots \arrow{r} & M_2 \arrow{r} & K \arrow{r} & 0
\end{tikzcd}
\end{equation}
in $\ce$ with all $M_i$ in $\cm$. Let $M$ be the $\ce$-acyclic complex obtained by splicing the 
complexes (\ref{eq:conflation}) and (\ref{eq:complex}), where the component $M_0$ is in degree $0$. 
Since all the objects $M_i$ lie in $\cm$ and $\cm$ is $d$-rigid, the complex $M^\wg=\cm(?,M)$
is a projective resolution of $Z$ as an $\cm$-module. Since the object $Z$ is right orthogonal to the subcategory $\ch^b(\cp)$ in $\ch^b(\cm)$, we can compute $\Ext^{d+1}_{\ol{\Ga}}(Z,Z)$ using the resolution $M^\wg$.
Let $\alpha$ be a class in $\Ext^{d+1}_{\ol{\Ga}}(Z,Z)$. It corresponds to a morphism $M^\wg \to \Si^{d+1} M^\wg$
in $\per \cm$ or equivalently, to a morphism $M \to \Si^{d+1} M$, still denoted by $\alpha$,
in the homotopy category $\ch^b(\cm)$.
Let $M' \subseteq M$ be the subcomplex $M_1 \to M_0$ and $M''$ the quotient $M/M'$. Notice that by the conflation (\ref{eq:conflation}), we have a
canonical quasi-isomorphism $\Si K \iso M'$. Consider the following diagram
\[
\begin{tikzcd}
M \arrow[swap]{d}{\alpha} \arrow{r}{\pi} & M'' \arrow{r}{s} \arrow[dashed]{d}{f} & \Si M' \arrow{r}{\iota} & \Si M \\
\Si^{d+1} M                 & \Si^{d+1} M'\arrow{l}{\Si^d \iota} & &
\end{tikzcd}
\]
in $\ch^b(\cm)$. By direct inspection (or using the canonical weight structure on $\ch^b(\cm)$), we see that the map
\[
\begin{tikzcd}
\ch^b(\cm)(M'', \Si^{d+1} M') \arrow{r} & \ch^b(\cm)(M, \Si^{d+1} M)
\end{tikzcd}
\]
taking $g$ to $\Si^d \iota \circ g \circ \pi$
is a bijection so that there is a unique morphism $f$ making the above square commutative. Since the complex $M$ is $\ce$-acyclic, we have a well-defined morphism $\beta = f\circ s^{-1}$ in
\[
\begin{tikzcd}
\cd^b(\ce)(\Si M',\Si^{d+1} M') \arrow[no head]{r}{\sim} & \cd^b(\ce)(\Si^2 K, \Si^{d+2} K) \arrow{r}{\sim} & \cc(\Si^2 K, \Si^{d+2} K) \: .
\end{tikzcd}
\]
The construction in the proof of part~b) of Theorem~5.4 of \cite{KellerReiten07} amounts to defining
\[
\phi_Z(\alpha) = \psi_{\Si^2 K}(\beta)\: .
\]
Now we have $\phi_Z(\alpha) = \phi_Z(\Si^d \iota \circ f \circ \pi)$.
This means that the linear form $\psi_{\Si^2 K}$ is obtained via Amiot's construction from the family of linear forms associated with $[\tilde{x}]$ and so by Remark~\ref{rk:compatibility}, we have
\[
\psi_{\Si^2 K}(\beta)=\tilde{\psi}_Z(\alpha) \: ,
\]
which ends the proof.
\end{proof}

By part~(c) of Theorem~1.5 of \cite{Keller98}, the exact sequence
\[
\begin{tikzcd}
0 \arrow{r} & \pvd\!_{dg}\ol{\Ga} \arrow{r} & \per\!_{dg}\ol{\Ga} \arrow{r} & \cc_{dg} \arrow{r} & 0
\end{tikzcd}
\]
of dg categories yields a long exact sequence in cyclic homology
\[
\begin{tikzcd}
HC_{-d}(\per\!_{dg}\ol{\Ga}) \arrow{r} & HC_{-d}(\cc_{dg}) \arrow{r}{\delta} & HC_{-d-1}(\pvd\!_{dg}\ol{\Ga}) \arrow{r} & HC_{-d-1}(\per\!_{dg}\ol{\Ga}) \: .
\end{tikzcd}
\]
Since the dg category $\per\!_{dg}\ol{\Ga}$ is derived Morita equivalent to the underlying dg algebra of $\ol{\Ga}$, their cyclic homologies are isomorphic. Since the dg algebra $\ol{\Ga}$ is connective and $d$ is a positive integer, the cyclic homologies $HC_{-d}(\per\!_{dg}\ol{\Ga})$ and $HC_{-d-1}(\per\!_{dg}\ol{\Ga})$ vanish. Therefore, the connecting morphism $\delta \colon HC_{-d}(\cc_{dg})\to HC_{-d-1}(\pvd\!_{dg}\ol{\Ga})$ is an isomorphism. By Proposition~\ref{prop:detect non-degeneracy}, the preimage of the given right $d$-Calabi--Yau structure on $\cc_{dg}$ under the dual map $D\delta \colon DHC_{-d-1}(\pvd\!_{dg}\ol{\Ga})\to DHC_{-d}(\cc_{dg})$ gives rise to a right $(d+1)$-Calabi--Yau structure on $\pvd\!_{dg}\ol{\Ga}$. Since the dg category $\pvd\!_{dg}\ol{\Ga}$ is quasi-equivalent to $\per\!_{dg}(\ol{\Ga}^!)$ and the dg category $\per\!_{dg}(\ol{\Ga}^!)$ is derived Morita equivalent to the dg algebra $\ol{\Ga}^!$, we obtain a right $(d+1)$-Calabi--Yau structure on $\ol{\Ga}^!$. By Proposition~\ref{prop:left and right}, its preimage under the $k$-linear bijection
\[
\begin{tikzcd}
HN_{d+1}(\ol{\Ga})\arrow{r}{\sim} & \D HC_{-d-1}(\ol{\Ga}^!)=DHC_{-d-1}(\ol{\Ga}^!)
\end{tikzcd}
\]
is a left $(d+1)$-Calabi--Yau structure on $\ol{\Ga}$. Then by the implication from (1) to (2) in Theorem~9.4 (also holds if the Calabi--Yau dimension is $2$) of \cite{VandenBergh15}, the $l$-augmented pseudo-compact dg algebra $\ol{\Ga}$ is quasi-isomorphic to a $(d+1)$-dimensional deformed dg preprojective algebra $\Pi$. A choice of quasi-isomorphism between the two yields a triangle equivalence between their perfect derived categories inducing a triangle equivalence between their perfectly valued derived categories. Thus, we have a triangle equivalence from $\cc_\Pi$ to $\cc$ taking the canonical $d$-cluster-tilting object in $\cc_\Pi$ to $T$. This concludes the proof of the implication from ii) to i) in Theorem~\ref{thm:main}.

\section{The structure theorem in the relative case} \label{section:relative}

\subsection{Assumptions used in the structure theorem} \label{ss:assumptions}

Let $k$ be a field of characteristic $0$ and $d$ a positive integer. Let $\ce$ be a small $k$-linear Frobenius exact pseudo-compact category which is Krull--Schmidt. We denote its stable category by $\cc$ and assume that it is Karoubian and $\Hom$-finite.
Assume that $\ce$ contains a $d$-cluster-tilting object $M$ and denote the full subcategory $\add M\subseteq \ce$ by $\cm$. Let $P$ be a maximal projective-injective direct summand of $M$. It is an additive generator for the full subcategory $\cp\subseteq \ce$ of the projective-injective objects in $\ce$. We denote the pseudo-compact algebra $\End_\ce(M)$ by $\Ga$ and assume that it is of finite global dimension. When $d$ equals $2$, a sufficient condition for this is given in Proposition~I.1.11 of \cite{BuanIyamaReitenScott09}. By the implication from ii) to i) in Corollary~\ref{cor:smooth}, the pseudo-compact algebra $\Ga$ is smooth.
We assume that the radical quotient $l$ of $\Ga$ (and hence the radical quotient $l_F$ of $\End_\cp(P)$) is finite-dimensional. Thus, by Remark~\ref{rk:Wedderburn}, the pseudo-compact algebra $\Ga$ is complete $l$-augmented and $\End_\cp(P)$ is complete $l_F$-augmented. Moreover, the morphism $\End_\cp(P)\to \End_\ce(M)$ between pseudo-compact algebras given by the inclusion of $P$ as a direct summand of $M$ induces an injective morphism $\phi\colon l_F \to l$ between algebras. Denote the quotient algebra $l/l_F$ by $\ol{l}$. Assume that we are given
\begin{itemize}
\item[-] a small smooth pretriangulated pseudo-compact dg category $\tilde{\cp}$ such that
$H^0(\tilde{\cp})$ contains a silting object $\tilde{P}$ and
\item[-] a continuous dg functor $\tilde{\cp}\to \cc^b_{dg}(\cp)$ such that the induced functor $H^0(\tilde{\cp})\to \ch^b(\cp)$ is essentially surjective and takes $\tilde{P}$ to the additive generator $P$ of $\cp \subseteq \ch^b(\cp)$.
\end{itemize}
We denote the pseudo-compact dg algebra $\End_{\tilde{\cp}}(\tilde{P})$ by $\Ga_F$. We assume that the continuous dg functor $\tilde{\cp} \to \cc^b_{dg}(\cp)$ induces an algebra isomorphism from the radical quotient of $H^0(\Ga_F)$ to the radical quotient $l_F$ of $\End_\cp(P)$. Thus, by Remark~\ref{rk:Wedderburn}, after possibly replacing $\tilde{\cp}$ with a quasi-equivalent pseudo-compact dg category, we may and will assume that the pseudo-compact dg algebra $\Ga_F$ is complete $l_F$-augmented and concentrated in non-positive degrees. Moreover, the morphism $\Ga_F \to \Ga$ between pseudo-compact dg algebras induced by the composed continuous dg functor
\[
\begin{tikzcd}
\tilde{\cp} \arrow{r} & \cc^b_{dg}(\cp) \arrow{r} & \cc^b_{dg}(\ce)
\end{tikzcd}
\]
is then $\phi$-augmented in the sense of section~\ref{ss:morphisms between pc algebras}.
 
The dg singularity category $\cc_{dg}=\cd^b_{dg}(\ce)/\cd^b_{dg}(\cp)$ is a dg enhancement of $\cc$. Notice that we have an equivalence $\cc^b_{dg}(\cm) \iso \per\!_{dg}\Ga$ of pseudo-compact dg categories taking $M$ to $\Ga$.
Using this equivalence and the dg functor $\tilde{\cp} \to \cc^b_{dg}(\cp)$, from the diagram~(\ref{eq:Palu's diagram 1}), we obtain a commutative diagram
\begin{equation} \label{eq:Palu's diagram 2}
\begin{tikzcd}
 & & \tilde{\cp} \arrow[equal]{r}\arrow{d} & \tilde{\cp}\arrow{d} &  \\
0 \arrow{r} & \pvd\!_{dg}\ol{\Ga} \arrow{r}\arrow[equal]{d} & \per\!_{dg}\Ga \arrow{r}\arrow{d} & \cd^b_{dg}(\ce)\arrow{r}\arrow{d} & 0 \\
0 \arrow{r} & \pvd\!_{dg}\ol{\Ga} \arrow{r} & \per\!_{dg}\ol{\Ga} \arrow{r}\arrow{d} & \cc_{dg}\arrow{r}\arrow{d} & 0 \\
 & & 0 & 0 &
\end{tikzcd}
\end{equation}
of underlying dg categories with exact rows and columns. By part~a) of Proposition~\ref{prop:pseudo-compact dg categories}, the dg derived category $\cd^b_{dg}(\ce)$ of $\ce$ carries a natural structure of pseudo-compact dg category. Since the pseudo-compact dg categories $\per\!_{dg}\Ga$ and $\cd^b_{dg}(\ce)$ are derived Morita equivalent to the pseudo-compact dg algebras $\Ga$ respectively $\RHom_\ce^{pc}(M,M)$, by Lemma~\ref{lem:pseudo-compact localisation}, the continuous dg functor $\per\!_{dg}\Ga \to \cd^b_{dg}(\ce)$ is a pseudo-compact localisation. By the variant of part~(c) of Proposition~3.10 of \cite{Keller11b} for pseudo-compact dg categories, the pseudo-compact dg category $\cd^b_{dg}(\ce)$ is smooth. In the statement of the following theorem, we use that 
the canonical map 
\begin{equation} \label{eq:canonical map}
\begin{tikzcd}
HN_{d+1}(\per\!_{dg}\Ga,\tilde{\cp})\arrow{r} & HN_{d+1}(\cd^b_{dg}(\ce),\tilde{\cp})
\end{tikzcd}
\end{equation}
is bijective. We will show this in section~\ref{ss:the canonical map is bijective}.

\subsection{The structure theorem}

Recall that, by Corollary~3.5.4 of \cite{KellerLiu23}, a left Calabi--Yau structure yields a canonical graded symplectic form on the $(-1)$-shifted dual of the graded Yoneda algebra of any perfectly valued dg module.

\begin{theorem} \label{thm:relative}
Let $k$ be a field of characteristic $0$ and $d$ a positive integer. Let $\ce$ be a small $k$-linear Frobenius exact pseudo-compact category which is Krull--Schmidt with the full subcategory $\cp$ of projective-injective objects. Assume that its stable category $\cc$ is \linebreak $\Hom$-finite. Then the following are equivalent.
\begin{itemize}
\item[i)] There exists a $(d+1)$-dimensional Ginzburg--Lazaroiu morphism
\[
\gamma=\gamma_{d+1}(\phi,N,F,\eta,w,w_F)
\]
and a triangle equivalence $\cc_\gamma \iso \cd^b(\ce)$ which yields an equivalence $\ch_\gamma \iso \ce$ of exact categories inducing an equivalence $\add l_F \Pi \iso \cp$, where $\l_F$ is the source of $\phi$ and $\Pi$ the target of $\gamma$.
\item[ii)] The assumptions in section~\ref{ss:assumptions} hold (in particular, the exact category $\ce$ contains a $d$-cluster-tilting object $M$) and we keep those notations. The continuous dg functor $\tilde{\cp}\to \cd^b_{dg}(\ce)$ in the diagram (\ref{eq:Palu's diagram 2}) carries a left relative $(d+1)$-Calabi--Yau structure whose image under the composed map
\[
\begin{tikzcd}
HN_{d+1}(\cd^b_{dg}(\ce),\tilde{\cp}) & HN_{d+1}(\per\!_{dg}\Ga,\tilde{\cp})\arrow[swap]{l}{\sim}\arrow{dl} & \\
HN_{d+1}(\per\!_{dg}\ol{\Ga})\arrow{r} & DHC_{-d-1}(\pvd\!_{dg}\ol{\Ga})\arrow{r}{D\delta} & DHC_{-d}(\cc_{dg})
\end{tikzcd}
\]
gives rise to a right $d$-Calabi--Yau structure on the dg category $\cc_{dg}$. Moreover, the kernel of the induced map 
\[
\begin{tikzcd}
\Si^{-1}\D\Ext^*_{\Ga_F}(l_F,l_F) \arrow{r} & \Si^{-1}\D\Ext^*_\Ga(l,l)
\end{tikzcd}
\]
is a Lagrangian homogeneous subspace concentrated in degrees less than or equal to $\frac{2-d}{2}$.
\end{itemize}
\end{theorem}

\begin{remark}
Thus, if statement ii) of the theorem holds, then the endomorphism algebra of $M$ is isomorphic to $H^0(\Pi)$.
\end{remark}

\subsection{Bijectivity of the canonical map (\ref{eq:canonical map})} \label{ss:the canonical map is bijective}

\begin{lemma} \label{lem:exact sequence of Koszul duals}
The sequence
\[
\begin{tikzcd}
0\arrow{r} & \pvd\!_{dg}(\cd^b_{dg}(\ce)) \arrow{r} & \pvd\!_{dg}(\per\!_{dg}\Ga) \arrow{r} & \pvd\!_{dg}(\pvd\!_{dg}\ol{\Ga}) \arrow{r} & 0
\end{tikzcd}
\]
of dg categories induced by the exact sequence in the middle row of the diagram (\ref{eq:Palu's diagram 2}) is exact.
\end{lemma}
\begin{proof}
The exact sequence in the middle row of the diagram (\ref{eq:Palu's diagram 2}) yields an exact sequence
\[
\begin{tikzcd}
0\arrow{r} & \cd(\cd^b_{dg}(\ce)) \arrow{r} & \cd(\per\!_{dg}\Ga) \arrow{r} & \cd(\pvd\!_{dg}\ol{\Ga}) \arrow{r} & 0
\end{tikzcd}
\]
of derived categories and restriction functors, where the localisation functor
\[
R\colon \cd(\per\!_{dg}\Ga)\longrightarrow \cd(\pvd\!_{dg}\ol{\Ga})
\]
has a fully faithful left adjoint $L$. We claim that the functor $L$ takes $\pvd(\pvd\!_{dg}\ol{\Ga})$ to $\pvd(\per\!_{dg}\Ga)$. Indeed, for any object
$X$ in $\pvd(\pvd\!_{dg}\ol{\Ga})$, we have $X \iso (R\circ L)X$ and this implies that the dg $\per\!_{dg}\Ga$-module $LX$ is perfectly valued because the triangulated category $\per\Ga$ has a single generator. Thus, the adjoint pair $(L,R)$ induces an adjoint pair between $\pvd(\pvd\!_{dg}\ol{\Ga})$ and $\pvd(\per\!_{dg}\Ga)$ so that the restriction of $R$ to $\pvd(\pvd\!_{dg}\ol{\Ga})$ is indeed a localisation functor. Since the middle row of the diagram (\ref{eq:Palu's diagram 2}) is exact, the kernel of the localisation functor induced by $R$ is $\pvd(\cd^b_{dg}(\ce))$. Therefore, the sequence
\[
\begin{tikzcd}
0\arrow{r} & \pvd(\cd^b_{dg}(\ce))\arrow{r} & \pvd(\per\!_{dg}\Ga)\arrow{r} & \pvd(\pvd\!_{dg}\ol{\Ga})\arrow{r} & 0
\end{tikzcd}
\]
of triangulated categories is exact and so is its dg enhancement.
\end{proof}

\begin{lemma} \label{lem:projective cover}
Any object $X$ in $\ce$ has a projective cover $P_X \to X$ which is a deflation.
\end{lemma}
\begin{proof}
For any object $X$ in $\ce$, since the category $\ce$ has enough projective objects, there is a deflation $P'\twoheadrightarrow X$ for some projective object $P'$. So we have a surjective morphism $\ce(P',P')\to \ce(P',X)$ of $\End_\ce(P')$-modules, which shows that $\ce(P',X)$ is a finitely generated $\End_\ce(P')$-module. Since $\ce$ is Krull--Schmidt, by Corollary~4.4 of \cite{Krause15}, the algebra $\End_\ce(P')$ is semi-perfect. This implies that $\ce(P',X)$ has a projective cover as an \linebreak $\End_\ce(P')$-module, which is of the form $\ce(P',P_X)\to \ce(P',X)$. Then the corresponding morphism $P_X\to X$ in $\ce$ is a projective cover of $X$. It is a deflation since it is a direct summand of the deflation $P'\twoheadrightarrow X$.
\end{proof}

\begin{lemma} \label{lem:augmented}
The pseudo-compact dg algebra $\RHom_\ce(M,M)$ is complete $l$-augmented.
\end{lemma}
\begin{proof}
By Lemma~\ref{lem:projective cover}, there exists a minimal projective resolution $\bp M$ of $M$. Denote its component of degree $-p$ by $P_M^p$ and the pseudo-compact dg algebra $\End_{\cc_{dg}(\ce)}(\bp M)$ by $A$. By construction, the differential of $\End_{\cc_{dg}(\cp/\rad \cp)}(\bp M)$ vanishes. So we have the canonical composed morphism
\[		
\End_{\cc_{dg}(\ce)}(\bp M)\longrightarrow \End_{\cc_{dg}(\cp/\rad \cp)}(\bp M)\longrightarrow \End_{\cc(\cp/\rad \cp)}(\bp M)\longrightarrow \End_{\cp/\rad \cp}(P_M^0)
\]
of pseudo-compact dg algebras. It fits into a commutative square of pseudo-compact dg algebras
\[
\begin{tikzcd}
\End_\ce(M)\arrow{r}\arrow{d} & \End_{\cc_{dg}(\ce)}(\bp M)\arrow{d} \\
\End_\ce(M)/\rad \End_\ce(M) \arrow{r} & \End_\ce(P_M^0)/\rad \End_\ce(P_M^0)\mathrlap{\: .}
\end{tikzcd}
\]
Since $P_M^0 \to M$ is a projective cover, the bottom morphism is an isomorphism. By Remark~\ref{rk:Wedderburn}, the left vertical morphism of pseudo-compact algebras admits a section. Therefore, the right vertical morphism makes $A$ into an $l$-augmented pseudo-compact dg algebra. Now we show it is complete. On the one hand, since $A$ is $l$-augmented, the dg $A$-module $l$ is semisimple. On the other hand, denote the subcomplex
\[
\prod_{n\in \Z}\{f\in A^n \mid f^p\in (\rad\ce)(P_M^p,P_M^{p+n})\}
\]
of $A$ by $I$. By construction, it is a dg ideal of $A$ and contained in $\rad A$. So any simple dg \linebreak $A$-module $S$ is annihilated by it and hence becomes a cyclic dg $A/I$-module. By construction, the differential of $A/I$ vanishes and therefore, so does the differential of $S$. So the dg $A$-module $S$ becomes a simple graded $H^*(A)$-module. Let $s$ be a nonzero homogeneous element of $S$.  Since $H^*(A)$ is concentrated in non-negative degrees, the subset $s\cdot H^{>0}(A)\subseteq S$ is a graded $H^*(A)$-submodule which does not contain the element $s$. Since $S$ is simple, this implies that $s\cdot H^{>0}(A)$ is zero and hence $S=s\cdot H^0(A)$ is concentrated in a single degree. So $S$ becomes a simple $H^0(A)=\Ga$-module. Since $\Ga$ is complete $l$-augmented, we deduce that the simple dg $A$-module $S$ is a direct summand of $l$. This ends the proof.
\end{proof}

\begin{proposition} \label{prop:bijectivity}
The canonical map (\ref{eq:canonical map}) is bijective.
\end{proposition}
\begin{proof}
Via derived Morita equivalences, the exact sequence
\[
\begin{tikzcd}
0 \arrow{r} & \pvd\!_{dg}\ol{\Ga} \arrow{r} & \per\!_{dg}\Ga \arrow{r} & \cd^b_{dg}(\ce) \arrow{r} & 0
\end{tikzcd}
\]
of dg categories corresponds to the exact sequence
\[
\begin{tikzcd}
0 \arrow{r} & \per\!_{dg}\La \arrow{r} & \per\!_{dg}\Ga \arrow{r} & \per\!_{dg}E \arrow{r} & 0\: ,
\end{tikzcd}
\]
where $\La=\RHom_{\ol{\Ga}}^{pc}(\ol{l},\ol{l})$ and $E=\RHom_\ce(M,M)$. If
we apply $\pvd\!_{dg}$ to this sequence, by Lemma~\ref{lem:exact sequence of Koszul duals}, we obtain another exact sequence
\[
\begin{tikzcd}
0 \arrow{r} & \pvd\!_{dg}E \arrow{r} & \pvd\!_{dg}\Ga \arrow{r} & \pvd\!_{dg}\La \arrow{r} & 0
\end{tikzcd}
\]
of dg categories. By Lemma~\ref{lem:augmented}, the Koszul dual $E^!$ of $E$ is well-defined and the dg category $\pvd\!_{dg}E$ is quasi-equivalent to $\per\!_{dg}(E^!)$. Moreover, the dg category $\pvd\!_{dg}\Ga$ is of course quasi-equivalent to $\per\!_{dg}(\Ga^!)$ and $\pvd\!_{dg}\La$ is quasi-equivalent to $\per\!_{dg}(\La^!)$.
Therefore, by part~(c) of Theorem~1.5 of \cite{Keller98} and derived Morita invariance of cyclic homology, there is a long exact sequence
\[
\begin{tikzcd}
HC_{-d}(\La^!) \arrow{r} & HC_{-d-1}(E^!) \arrow{r} & HC_{-d-1}(\Ga^!) \arrow{r} & HC_{-d-1}(\La^!)
\end{tikzcd}
\]
in cyclic homology. Since the dg algebra $\ol{\Ga}$ is connective, so is $\La^!$. Since $d$ is a positive integer, this implies that the cyclic homologies $HC_{-d}(\La^!)$ and $HC_{-d-1}(\La^!)$ vanish. Therefore, the canonical map 
\[
HC_{-d-1}(E^!) \longrightarrow HC_{-d-1}(\Ga^!)
\]
is bijective. By section~\ref{ss:left and right}, there is a canonical bijection $HN_{*}(A)\iso \D HC_{-*}(A^!)$ of pseudo-compact graded vector spaces for any complete augmented pseudo-compact dg algebra $A$. 
So the canonical map
\[
\begin{tikzcd}
HN_{d+1}(\per\!_{dg}\Ga,\tilde{\cp})\arrow{r} & HN_{d+1}(\cd^b_{dg}(\ce),\tilde{\cp})
\end{tikzcd}
\]
is bijective.
\end{proof}

\subsection{Proof of the implication from i) to ii) in Theorem~\ref{thm:relative}} \label{ss:proof from i) to ii) in relative}

Without loss of generality, we may and will assume that $\ce$ equals $\ch_\gamma$. Since the category
\[
\begin{tikzcd}
\cc \arrow{r}{\sim} & \cd^b(\ce)/\cd^b(\cp) = \cc_\gamma / \cd^b(\add l_F \Pi) \arrow{r}{\sim} & \cc_{\Pi/(l_F)}
\end{tikzcd}
\]
is $\Hom$-finite, by the implication from i) to ii) in Theorem~\ref{thm:main}, the preimage of $\Pi/(l_F)$ in $\cc$ is a $d$-cluster-tilting object. So $\Pi$ is a $d$-cluster-tilting object in $\ce$. Denote the source of the Ginzburg--Lazaroiu morphism $\gamma$ by $\Pi_F$. The pseudo-compact dg category $\tilde{\cp}=\per\!_{dg}\Pi_F$ and the canonical continuous dg functor $\per\!_{dg}\Pi_F \to \cc^b_{dg}(l_F \Pi)$ which sends $\Pi_F$ to $l_F \Pi$ satisfy the assumptions in section~\ref{ss:assumptions}. By the implication from i) to ii) in Theorem~4.3.1 of \cite{KellerLiu23}, the $(d+1)$-dimensional Ginzburg--Lazaroiu morphism $\gamma$ carries a left relative $(d+1)$-Calabi--Yau structure and the kernel of the induced map $\Si^{-1}\D\Ext^*_{\Pi_F}(l_F,l_F) \to \Si^{-1}\D\Ext^*_\Pi(l,l)$ is a Lagrangian homogeneous subspace concentrated in degrees less than or equal to $\frac{2-d}{2}$. Since the pseudo-compact dg categories $\per\!_{dg}\Pi$ and $\per\!_{dg}\Pi_F$ are derived Morita equivalent to the pseudo-compact dg algebras $\Pi$ respectively $\Pi_F$, the canonical continuous dg functor $\per\!_{dg}\Pi_F \to \per\!_{dg}\Pi$ also carries a left relative $(d+1)$-Calabi--Yau structure $[\tilde{\xi}]$. Since the composed continuous functor
\[
\begin{tikzcd}
\per\!_{dg}\Pi_F \arrow{r} & \cc^b_{dg}(l_F \Pi)\arrow{r} & \cd^b_{dg}(\ch_\gamma)
\end{tikzcd}
\]
is the composition of the pseudo-compact localisation $\per\!_{dg}\Pi \to \cc_{\gamma,dg}$ with the canonical continuous dg functor $\per\!_{dg}\Pi_F \to \per\!_{dg}\Pi$ induced by $\gamma$, by the variant of part~(c) of Proposition~3.10 of \cite{Keller11b} for pseudo-compact dg categories, the image of $[\tilde{\xi}]$ under the map
\[
\begin{tikzcd}
HN_{d+1}(\per\!_{dg}\Pi,\per\!_{dg}\Pi_F)\arrow{r} & HN_{d+1}(\cd^b_{dg}(\ch_\gamma),\per\!_{dg}\Pi_F)
\end{tikzcd}
\]
is a left relative $(d+1)$-Calabi--Yau structure on the continuous dg functor
\[
\per\!_{dg}\Pi_F\to \cd^b_{dg}(\ch_\gamma)\: .
\]
By the variant of Corollary~7.1 of \cite{BravDyckerhoff19} for pseudo-compact dg categories, the image of $[\tilde{\xi}]$ under the map
\[
\begin{tikzcd}
HN_{d+1}(\per\!_{dg}\Pi,\per\!_{dg}\Pi_F)\arrow{r} & HN_{d+1}(\per\!_{dg}(\Pi/(l_F)))
\end{tikzcd}
\]
is a left $(d+1)$-Calabi--Yau structure on the homotopy cofibre $\per\!_{dg}(\Pi/(l_F))$. Then by the proof of the implication from i) to ii) in Theorem~\ref{thm:main}, the image of the left \linebreak $(d+1)$-Calabi--Yau structure on $\per\!_{dg}(\Pi/(l_F))$ under the composed map
\[
\begin{tikzcd}
HN_{d+1}(\per\!_{dg}(\Pi/(l_F)))\arrow{r} & DHC_{-d-1}(\pvd\!_{dg}(\Pi/(l_F)))\arrow{r}{D\delta} & DHC_{-d}(\cc_{\Pi/(l_F),dg})
\end{tikzcd}
\]
is a right $d$-Calabi--Yau structure on the dg category $\cc_{\Pi/(l_F),dg}$. This concludes the proof of the implication from i) to ii) in Theorem~\ref{thm:relative}.

\subsection{Proof of the implication from ii) to i) in Theorem~\ref{thm:relative}} \label{ss:proof from ii) to i) in relative}

Let $[\tilde{\xi}]$ be a class in $HN_{d+1}(\per\!_{dg}\Ga,\tilde{\cp})$ such that its images $[\xi]$ in $HN_{d+1}(\cd^b_{dg}(\ce),\tilde{\cp})$ and $[x]$ in $DHC_{-d}(\cc_{dg})$ are non-degenerate. We want to prove that $[\tilde{\xi}]$ is non-degenerate, which means that the morphism $\Si^d\tilde{\cp}^\vee \to \tilde{\cp}$ in $\cd(\tilde{\cp}^e)$ and the morphism
\[
\begin{tikzcd}
\Si^{d+1}\cone(\mu \colon \per\!_{dg}\Ga\lten_{\tilde{\cp}}\per\!_{dg}\Ga \to \per\!_{dg}\Ga)^\vee \arrow{r} & \per\!_{dg}\Ga
\end{tikzcd}
\]
in $\cd((\per\!_{dg}\Ga)^e)$ (here we consider $\per\!_{dg}\Ga$ as a pseudo-compact dg category) obtained from the image of $[\tilde{\xi}]$ under the canonical map
\[
\begin{tikzcd}
HN_{d+1}(\per\!_{dg}\Ga,\tilde{\cp})\arrow{r} & HH_{d+1}(\per\!_{dg}\Ga,\tilde{\cp})
\end{tikzcd}
\]
are isomorphisms. The first one follows from the non-degeneracy of $[\xi]$. To prove the second one, it suffices to prove that its image
\begin{equation} \label{eq:isomorphism 1}
\begin{tikzcd}
L\lten_{\per\!_{dg}\Ga}\Si^{d+1}\cone(\mu)^\vee \lten_{\per\!_{dg}\Ga}M \arrow{r} & L\lten_{\per\!_{dg}\Ga}\per\!_{dg}\Ga\lten_{\per\!_{dg}\Ga}M
\end{tikzcd}
\end{equation}
under the functor $L\lten_{\per\!_{dg}\Ga}\, ?\lten_{\per\!_{dg}\Ga}M$ is an isomorphism for all $L\in \cd(\per\!_{dg}\Ga)$ and $M\in \cd((\per\!_{dg}\Ga)\op)$. We first consider the case that $L$ lies in the image of the restriction functor
$\cd(\cd^b_{dg}(\ce))\to \cd(\per\!_{dg}\Ga)$ and
$M$ lies in that of $\cd((\cd^b_{dg}(\ce))\op)\to \cd((\per\!_{dg}\Ga)\op)$. Since the class $[\xi]$ in $HN_{d+1}(\cd^b_{dg}(\ce),\tilde{\cp})$ is non-degenerate, the morphism
\begin{equation} \label{eq:isomorphism 2}
\begin{tikzcd}
\Si^{d+1}\cone(\mu \colon \cd^b_{dg}(\ce)\lten_{\tilde{\cp}}\cd^b_{dg}(\ce)\to \cd^b_{dg}(\ce))^\vee\arrow{r} & \cd^b_{dg}(\ce)
\end{tikzcd}
\end{equation}
in $\cd((\cd^b_{dg}(\ce))^e)$ obtained from the image of $[\xi]$ under the canonical map
\[
\begin{tikzcd}
HN_{d+1}(\cd^b_{dg}(\ce),\tilde{\cp})\arrow{r} & HH_{d+1}(\cd^b_{dg}(\ce),\tilde{\cp})
\end{tikzcd}
\]
is an isomorphism. Since the continuous dg functor $\per\!_{dg}\Ga \to \cd^b_{dg}(\ce)$ is a pseudo-compact localisation, the morphism $\mu \colon \cd^b_{dg}(\ce)\lten_{\per\!_{dg}\Ga}\cd^b_{dg}(\ce)\to \cd^b_{dg}(\ce)$ in $\cd((\cd^b_{dg}(\ce))^e)$ is an isomorphism. Then if we apply the functor $L\lten_{\cd^b_{dg}(\ce)}\, ?\lten_{\cd^b_{dg}(\ce)}M$ to the isomorphism (\ref{eq:isomorphism 2}), we obtain the isomorphism (\ref{eq:isomorphism 1}).
Now we consider the case that $M$ lies in the image of the induction functor $\cd((\pvd\!_{dg}\ol{\Ga})\op)\to \cd((\per\!_{dg}\Ga)\op)$, which is the colocalising subcategory generated by the simple pseudo-compact dg $(\per\!_{dg}\Ga)\op$-modules which restrict to zero on $\tilde{\cp}$. In this case, the morphism (\ref{eq:isomorphism 1}) reduces to
\[
\begin{tikzcd}
L\lten_{\per\!_{dg}\Ga}\Si^{d+1}(\per\!_{dg}\Ga)^\vee \lten_{\per\!_{dg}\Ga}M \arrow{r} & L\lten_{\per\!_{dg}\Ga}\per\!_{dg}\Ga\lten_{\per\!_{dg}\Ga}M\: .
\end{tikzcd}
\]
Since the pseudo-compact dg category $\per\!_{dg}\Ga$ is derived Morita equivalent to $\Ga$, replacing $\per\!_{dg}\Ga$ by $\Ga$ we reduce it to $L\lten_\Ga \Si^{d+1}\Ga^\vee \lten_\Ga M \to L\lten_\Ga \Ga\lten_\Ga M$. Because the derived category $\cd(\Ga)$ equals the colocalising subcategory generated by $\per\Ga$ and derived pseudo-compact tensor functors commute with arbitrary products, it suffices to consider the case that $L$ is a perfect pseudo-compact dg $\Ga$-module and $M$ is a simple pseudo-compact dg $\ol{\Ga}\op$-module. Since there are isomorphisms
\[
\D(L\lten_\Ga \Si^{d+1}\Ga^\vee \lten_\Ga M) \simeq \RHom_\Ga(L\lten_\Ga \Si^{d+1}\Ga^\vee,DM) \iso D\RHom_\Ga(DM,\Si^{d+1}L)
\]
and
\[
\begin{tikzcd}
\D(L\lten_\Ga \Ga\lten_\Ga M)\arrow[no head]{r}{\sim} & \RHom_\Ga(L\lten_\Ga \Ga,DM) & \RHom_\Ga(L,DM)\arrow[swap]{l}{\sim}\: ,
\end{tikzcd}
\]
it suffices to prove that the morphism
\begin{equation} \label{eq: d+1 non-degeneracy 2}
\begin{tikzcd}
\RHom_\Ga(L,DM)\arrow{r} & D\RHom_\Ga(DM,\Si^{d+1}L)
\end{tikzcd}
\end{equation}
in $\cd(k)$ is an isomorphism. By the bifunctoriality, the induced morphism in homology is then determined by the linear forms
\[
\tilde{\psi}_{DM}\colon \Ext^{d+1}_\Ga(DM,DM)\longrightarrow k
\]
associated with the preimage of $[x]$ in $DHC_{-d-1}(\pvd\!_{dg}\ol{\Ga})$. On the other hand, from part~c) of Theorem~5.4 of \cite{KellerReiten07}, we obtain a functorial isomorphism
\[
\begin{tikzcd}
(\per \Ga)(DM,?)\arrow{r}{\sim} & D(\per \Ga)(?,\Si^{d+1}DM)\: .
\end{tikzcd}
\]
It is completely determined by the linear form
\[
\phi_{DM}\colon \Ext^{d+1}_\Ga (DM,DM) \longrightarrow k
\]
obtained as the image of the identity of $DM$. By the proof of Proposition~\ref{prop:detect non-degeneracy}, we have $\phi_{DM}=\tilde{\psi}_{DM}$. This implies that the morphism (\ref{eq: d+1 non-degeneracy 2}) is an isomorphism, which ends the proof in this case. Similarly, we can prove that the morphism (\ref{eq:isomorphism 1}) is an isomorphism in the case that $L$ lies in the image of the induction functor $\cd(\pvd\!_{dg}\ol{\Ga})\to \cd(\per\!_{dg}\Ga)$. Finally, since the exact sequence of pseudo-compact dg categories
\[
\begin{tikzcd}
0 \arrow{r} & \pvd\!_{dg}\ol{\Ga}\arrow{r} & \per\!_{dg}\Ga \arrow{r} & \cd^b_{dg}(\ce)\arrow{r} & 0
\end{tikzcd}
\]
yields a semi-orthogonal decomposition of the derived category $\cd(\per\!_{dg}\Ga)$, the general case can be reduced to the above cases. In conclusion, the class $[\tilde{\xi}]$ in $HN_{d+1}(\per\!_{dg}\Ga,\tilde{\cp})$ is non-degenerate. 
Since the pseudo-compact dg categories $\per\!_{dg}\Ga$ and $\tilde{\cp}$ are derived Morita equivalent to $\Ga$ respectively $\Ga_F$, we find that the morphism $\Ga_F \to \Ga$ between pseudo-compact dg algebras also carries a left relative $(d+1)$-Calabi--Yau structure. Then by the implication from ii) to i) in Theorem~4.3.1 of \cite{KellerLiu23}, the $\phi$-augmented morphism $\Ga_F \to \Ga$ between pseudo-compact dg algebras is weakly equivalent to a $(d+1)$-dimensional Ginzburg--Lazaroiu morphism $\gamma$ with target $\Pi$. A choice of quasi-isomorphism between the sources and targets yields triangle equivalences between their perfect derived categories inducing triangle equivalences between their perfectly valued derived categories. Thus, we have a triangle equivalence from $\cc_\gamma$ to $\cd^b(\ce)$ taking $\Pi$ to $M$ and $l_F \Pi$ to $P$. Next we show that the image $\ce'$ of the subcategory $\ch_\gamma \subseteq \cc_\gamma$ under the equivalence $\cc_\gamma \iso \cd^b(\ce)$ is equivalent to the category $\ce$. If $d$ equals $1$, then it is obvious. If $d$ is greater than or equal to $2$, since the stable category $\cc$ is $\Hom$-finite, for any object $X$ in $\ce$, the vector spaces $\cd^b(\ce)(M,\Si^i X)\iso \cc(M,\Si^i X)$ are finite-dimensional for all $i=1$, $\ldots$, $d-1$. Since $M$ is a $d$-cluster-tilting object in $\ce$, by the variant of the implication from (1) to (2-0) in Proposition~2.4.1 of \cite{Iyama07} for exact categories, there exist conflations
\[
\begin{tikzcd}
0 \arrow{r} & X_1 \arrow{r}{f_1} & M_0 \arrow{r} & X \arrow{r} & 0\: ,
\end{tikzcd}
\]
\[
\begin{tikzcd}
0 \arrow{r} & X_2 \arrow{r}{f_2} & M_1 \arrow{r} & X_1 \arrow{r} & 0\: ,
\end{tikzcd}
\]
\[
\ldots
\]
\[
\begin{tikzcd}
0 \arrow{r} & X_{d-2} \arrow{r}{f_{d-2}} & M_{d-3} \arrow{r} & X_{d-3} \arrow{r} & 0\: ,
\end{tikzcd}
\]
\[
\begin{tikzcd}
0 \arrow{r} & M_{d-1} \arrow{r}{f_{d-1}} & M_{d-2} \arrow{r} & X_{d-2} \arrow{r} & 0
\end{tikzcd}
\]
in $\ce$ with all $M_i$ in $\cm$. Since all $f_i$ are inflations, all the maps $\ce(f_i,P)$ are surjective. So these conflations give rise to triangles in $\cd^b(\ce)$ with all $M_i$ in $\cm$ and all $\cd^b(\ce)(f_i,P)$ surjective. Similarly in the dual case. This shows that the category $\ce$ is contained in $\ce'$. On the other hand, for any object $X$ in $\ce'$, there exist triangles
\[
\begin{tikzcd}
X_1 \arrow{r}{f_1} & M_0 \arrow{r} & X \arrow{r} & \Si X_1\: ,
\end{tikzcd}
\]
\[
\begin{tikzcd}
X_2 \arrow{r}{f_2} & M_1 \arrow{r} & X_1 \arrow{r} & \Si X_2\: ,
\end{tikzcd}
\]
\[
\ldots
\]
\[
\begin{tikzcd}
X_{d-2} \arrow{r}{f_{d-2}} & M_{d-3} \arrow{r} & X_{d-3} \arrow{r} & \Si X_{d-2}\: ,
\end{tikzcd}
\]
\[
\begin{tikzcd}
M_{d-1} \arrow{r}{f_{d-1}} & M_{d-2} \arrow{r} & X_{d-2} \arrow{r} & \Si M_{d-1}
\end{tikzcd}
\]
in $\cd^b(\ce)$ with all $M_i$ in $\cm$ and all $\cd^b(\ce)(f_i,P)$ surjective. Since the objects $M_{d-2}$ and $M_{d-1}$ lie in $\ce$, the latter condition implies that $f_{d-1}$ can be assumed to be an inflation in $\ce$. Thus, the object $X_{d-2}$ can be assumed to lie in $\ce$. By induction, we deduce that $X$ is isomorphic to an object in $\ce$. This shows that the category $\ce'$ is contained in the closure under isomorphisms of the subcategory $\ce \subseteq \cd^b(\ce)$. Thus, the triangle equivalence $\cc_\gamma\iso \cd^b(\ce)$ yields an equivalence $\ch_\gamma \iso \ce$ of extriangulated categories. Since the category $\ce$ is exact, so is $\ch_\gamma$ and the equivalence $\ch_\gamma \iso \ce$ is an equivalence of exact categories. This concludes the proof of the implication from ii) to i) in Theorem~\ref{thm:relative}.



\def\cprime{$'$} \def\cprime{$'$}
\providecommand{\bysame}{\leavevmode\hbox to3em{\hrulefill}\thinspace}
\providecommand{\MR}{\relax\ifhmode\unskip\space\fi MR }
\providecommand{\MRhref}[2]{%
  \href{http://www.ams.org/mathscinet-getitem?mr=#1}{#2}
}
\providecommand{\href}[2]{#2}


\end{document}